\documentclass[12pt]{amsart}
\usepackage{palatino}
\usepackage{amsfonts}
\usepackage{amssymb}
\usepackage{amscd}
\usepackage{pdfpages}
\usepackage{xypic}
\usepackage{color}
\usepackage{graphicx}

\newcommand{\dblq}{{/\!/}}
\newtheorem{theorem}{Theorem}
\newtheorem{proposition}[theorem]{Proposition}

\newtheorem{lemma}[theorem]{Lemma}
\theoremstyle{definition}
\newtheorem{definition}[theorem]{Definition}

\newtheorem{conjecture/question}[theorem]{Conjecture/Question}

\newtheorem{remark/definition}[theorem]{Remark/Definition}
\newtheorem{terminology/notation}[theorem]{Terminology/Notation}

\def\OO{\mathcal{O}}

\def\cB{\mathcal{B}}

\def\cW{\mathcal{W}}

\def\cS{\mathcal{S}}

\def\cM{\mathcal{M}}

\def\H{\mathcal{H}}
\def\Pic0{{\rm Pic}^0(X)}

\def\mm{\overline{\mathcal{M}}}
\def\ss{\overline{\mathcal{S}}}
\def\hh{\overline{\mathcal{H}}}

\def\oM{\overline{\mathcal{M}}}

\def\EE{\mathrm{E}}
\def\n{\mathrm{n}}
\def\LL{\mathrm{L}}
\def\VV{\mathrm{V}}
\def\HH{\mathrm{H}}
\def\g{\mathrm{g}}
\def\rarr{\rightarrow}
\def\Z{\mathbb{Z}}
\def\XC{\mathcal{O}_{\mathcal{X}}}

\def\PPP{\mathsf{P}}

\def\w={\stackrel{\sim}{=}}

\begin{document}
\title{The moduli space of twisted canonical divisors}

\author[G. Farkas]{Gavril Farkas}

\address{Humboldt-Universit\"at zu Berlin, Institut F\"ur Mathematik
\hfill \newline\texttt{}
 \indent Unter den Linden 6, 10099 Berlin, Germany} \email{{\tt farkas@math.hu-berlin.de}}
\thanks{}

\author[R. Pandharipande]{Rahul Pandharipande}
\address{ETH Z\"urich, Department of Mathematics, Raemistrasse 101 \hfill
 \newline \indent   8092 Z\"urich, Switzerland }
 \email{{\tt rahul@math.ethz.ch}}

\date{October 2015}

\pagestyle{plain}

\begin{abstract}
The moduli space of canonical divisors (with prescribed zeros and poles)
on nonsingular curves is not
compact since the curve may degenerate. We define a proper moduli space of twisted
canonical divisors in $\overline{\mathcal{M}}_{g,n}$
which includes the space of canonical divisors as
an open subset.
The theory leads to geometric/combinatorial constraints on
the closures of the moduli spaces of canonical divisors.

In case the differentials have at least one pole (the strictly
meromorphic case), the moduli spaces of twisted canonical divisors on genus $g$ curves
are of pure codimension $g$ in $\overline{\mathcal{M}}_{g,n}$.
In addition to the closure of
the canonical divisors on nonsingular curves, the moduli spaces have virtual components.
In the Appendix, a complete proposal relating the sum of the fundamental classes
of all components (with intrinsic multiplicities) to a formula of
Pixton is proposed. The result is a precise and explicit conjecture in the
tautological ring for the
weighted fundamental class of the moduli spaces of twisted canonical divisors.

As a consequence of the conjecture, the classes of the closures of
the moduli spaces of canonical divisors on nonsingular curves are determined
(both in the holomorphic and meromorphic cases).
\end{abstract}
\maketitle

\setcounter{tocdepth}{1}
\tableofcontents

\newpage
\setcounter{section}{-1}
\section{Introduction}

\subsection{Zeros and poles} \label{zp}
Let $\cM_g$ be the moduli space of nonsingular curves of genus $g\geq 2$. The
{\em Hodge bundle},
$$\mathbb{E} \rightarrow \cM_g\, ,$$
has fiber
over the moduli point $[C]\in \cM_g$ given by the space of
{\em holomorphic differentials} $H^0(C,\omega_C)$.
The projectivization
$$\H_g = \mathbb{P}(\mathbb{E}) \rightarrow \cM_g$$
is a moduli space of {\em canonical divisors} on nonsingular curves.
We may further stratify $\H_g$ by canonical divisors with
multiplicities of zeros
specified by a partition
$\mu$ of $2g-2$.
Neither $\H_g$ nor the strata with specified zero multiplicities
 are compact.
We
describe a natural compactification of the strata associated
to a partition $\mu$.

{\em Meromorphic differentials}
arise naturally in the analysis of the boundary of the
spaces of holomorphic differentials.
We will  consider meromorphic differentials on curves with
prescribed zero and pole multiplicities.
Let
$$\mu=(m_1, \ldots, m_n)\,, \ \ \ \ m_i \in \mathbb{Z}$$
satisfy $\sum_{i=1}^n m_i=2g-2$.
The vector $\mu$ prescribes the
zero multiplicities of a meromorphic differential
via the positive parts $m_i>0$
and the pole multiplicities via the negative parts $m_i<0$.
Parts with $m_i=0$ are also permitted (and will correspond
to points on the curve which are neither zeros nor poles
and otherwise unconstrained).

Let $g$ and $n$ be in the stable range $2g-2+n>0$. For a vector $\mu$
of length $n$,
we define the closed substack  $\H_g(\mu)\subset \cM_{g,n}$ by
$$\H_g(\mu)=\Bigl\{[C, p_1, \ldots, p_n]\in \cM_{g,n}\, \Big| \, \OO_C\Bigl(\sum_{i=1}^n m_i p_i\Bigr)=\omega_C \Bigr\}\, . $$
Consider first the case where
 all parts of $\mu$ are non-negative.
Since $\H_g(\mu)$ is the locus of points
$$[C, p_1, \ldots, p_n]\in \cM_{g,n}$$ for which the evaluation map
$$H^0(C, \omega_C)\rightarrow H^0\bigl(C, \omega_{C | m_1p_1+\cdots+m_n p_n}\bigr)$$ is not injective, every component of $\H_g(\mu)$ has dimension at least
 $2g-2+n$
in $\cM_{g,n}$ by degeneracy loci considerations \cite{Fulton}.
Polishchuk \cite{Pol} has shown
that $\H_g(\mu)$ is a \emph{nonsingular} substack of $\cM_{g,n}$ of
pure dimension $2g-2+n$. In fact, the arguments of
 \cite{Pol} can be extended to
the case where the vector $\mu$ has negative
parts.{\footnote{Occurances of $\omega_C$
in Polishchuk's argument should be
replaced by $\omega_C\left(\sum_{i=1}^{\widehat{n}} |m_i|p_i\right)$
where $m_1,...,m_{\widehat{n}}$ are the negative parts of $\mu$.
The reason for the change from codimension $g-1$ in the holomorphic case
to codimension $g$ in the strictly meromorphic case is that  $\omega_C$
is special and $\omega_C\left(\sum_{i=1}^{\widehat{n}} |m_i|p_i\right)$ is not.}}
In the strictly meromorphic case, $\H_g(\mu)$ is a nonsingular substack of dimension $2g-3+n$ in $\cM_{g,n}$.


In the holomorphic case,
the spaces $\H_g(\mu)$ have been intensely studied from the point of view of flat surfaces (which leads to
natural coordinates and a volume form), see  \cite{EMZ}.
However, an algebro-geometric study of the strata has been neglected:
basic questions concerning the birational geometry and the cohomology classes of the strata of differentials are open.

We define here  a proper moduli space of {\em twisted canonical
divisors}
$$\widetilde{\H}_g(\mu)\subset \mm_{g,n}$$ which contains $\H_g(\mu)$ as
an open set
$$\H_g(\mu) \subset \widetilde{\H}_g(\mu)\, .$$
The space $\widetilde{\H}_g(\mu)$ will typically be {\em larger} than the
closure
$$\overline{\H}_g(\mu) \subset \widetilde{\H}_g(\mu)\, .$$
We prove every irreducible component of $\widetilde{\H}_g(\mu)$ supported
entirely in the boundary of $\overline{\mathcal{M}}_{g,n}$ has
codimension $g$ in $\mm_{g,n}$.

In the {\em strictly meromorphic case} where there exists an $m_i<0$,  the moduli space
$\widetilde{\H}_g(\mu)$
is of pure codimension $g$ in $\overline{\mathcal{M}}_{g,n}$.  The closure
$${\hh}_g(\mu) \subset \widetilde{\H}_g(\mu)$$
is a union of irreducible components.
The components of $\widetilde{\H}_g(\mu)$ which do not lie in ${\hh}_g(\mu)$
are called {\em virtual}. The virtual components play a basic
 role in the study of canonical divisors.

\subsection{Pixton's formula}
A relation to a beautiful formula of Pixton
for an associated cycle class in $R^g(\overline{\mathcal{M}}_{g,n})$
 is explored in the Appendix: the
contributions of the virtual
components of $\widetilde{\H}_g(\mu)$
are required to match Pixton's formula.
The result is
a precise and explicit conjecture in the tautological ring of $\overline{\mathcal{M}}_{g,n}$ for the
 sum of the  fundamental classes (with intrinsic multiplicities)
of all irreducible components of $\widetilde{\H}_g(\mu)$
in the strictly meromorphic case.

In the holomorphic case where all $m_i\geq 0$, the connection to
Pixton's formula
is more subtle since $$\widetilde{\H}_g(\mu)\subset \overline{M}_{g,n}$$ is not of pure codimension $g$.
The more refined
 approach to the moduli of canonical divisors proposed by Janda
\cite{J} via relative Gromov-Witten theory and virtual fundamental classes
will likely be required to understand the holomorphic case
(and to prove the conjecture of the Appendix in the strictly meromorphic case).
However, the virtual components clearly also play a role in the
holomorphic case.

Remarkably, the fundamental classes of the varieties of closures
$$\overline{\H}_g(\mu) \subset \overline{\mathcal{M}}_{g,n}\, ,$$
in {\em both} the holomorphic and strictly meromorphic cases,
are determined in the Appendix as a consequence of the
conjectured link between Pixton's formula and
the weighted fundamental class of the moduli space of twisted canonical
divisors in the strictly meromorphic case.

\subsection{Twists}\label{twww}
Let $[C, p_1,\ldots, p_n]$ be a Deligne-Mumford stable $n$-pointed
curve.{\footnote{The definition of a twist given here is valid for
any connected nodal curve.}}
Let $$\mathsf{N}(C)\subset C$$
be the nodal locus.
A node $q \in \mathsf{N}(C)$ is {\em basic} if
$q$ lies in the intersection of two distinct
irreducible components of $C$. Let $$\mathsf{BN}(C)\subset \mathsf{N}(C)$$
be the set of basic nodes, and let
$$\widetilde{\mathsf{BN}}(C) \rightarrow \mathsf{BN}(C)$$
be the canonical double cover defined by
$$\widetilde{\mathsf{BN}}(C)= \big\{ \, (q,D)\, | \, q\in \mathsf{BN}(C)\, ,\
q\in D\, ,
  \ \text{and}  \ D\subset C \  \text{an irreducible component} \big\} \, .
$$
A \emph{twist} $I$ of the curve $C$ is a function
$$I : \widetilde{\mathsf{BN}}(C) \rightarrow \mathbb{Z}\, $$
satisfying the {\em balancing}, {\em vanishing}, {\em sign}, and
{\em transitivity} conditions defined as follows.

\vspace{5pt}
\noindent  {\bf{Balancing:}} \emph{If a basic node $q$ lies in
the intersection of distinct irreducible components $D,D'\subset  C$, then
$$I(q, D)+I(q,D')=0\, .$$ }

Let $\mathsf{Irr}(C)$ be the set of irreducible components of $C$.
If $$D,D'\in \mathsf{Irr}(C)$$ and there exists $q\in D\cap D'$ with
$I(q,D)=0$, we write $$D\approx D'\, .$$ By the balancing
condition, $\approx$ is symmetric.
Let $\sim$ be the minimal {\em equivalence relation} on $\mathsf{Irr}(C)$
generated by $\approx$.

\vspace{5pt}
\noindent {\bf{Vanishing:}} \emph{If $D,D'\in \mathsf{Irr}(C)$ are distinct
irreducible components
in the same $\sim$-equivalence class and $q\in D\cap D'$, then
$$I(q,D)=I(q,D')=0\, .$$ }
\vspace{5pt}

\noindent {\bf{Sign:}}\emph{
Let $q\in D \cap D'$ and
$\widehat{q} \in \widehat{D} \cap \widehat{D}'$ be
basic nodes of $C$. If
$ D\sim \widehat{D}$ and $ D'\sim \widehat{D}'$, then
$$I({q},{D})>0 \ \Longrightarrow\
I (\widehat{q}, \widehat{D})>0\, .$$}

\vspace{5pt}

Let $\mathsf{Irr}(C)^\sim$ be the set of $\sim$-equivalence
classes.
We define
a directed graph $\Gamma_I(C)$ with vertex set $\mathsf{Irr}(C)^\sim$
by the following construction.
A directed edge
$$v\rightarrow v'$$
is placed between equivalence classes $v,v'\in \mathsf{Irr}(C)^\sim$
if there exist
\begin{equation}\label{cggc}
D\in v\, ,\ \ D' \in v'\,\ \ \text{and}\ \  q\in D\cap D'
\end{equation}
satisfying $I(q,D)>0.$

The vanishing condition prohibits self-edges at the vertices of $\Gamma_I(C)$.
By the sign condition, vertices are {\em not} connected
by directed edges in both directions: there is at most
a single directed edge between vertices.
The fourth defining condition for a  twist $I$ is easily stated
in terms of the associated graph $\Gamma_I(C)$.

\vspace{5pt}

\noindent  {\bf{Transitivity:}}
\emph{The graph $\Gamma_I(C)$
has no directed loops.}

\vspace{10pt}

If $C$ is curve of compact type, distinct irreducible components
of $C$ intersect in at most one point. The vanishing and sign conditions
are therefore trivial. Since $\Gamma_I(C)$ is a tree, the
transitivity condition is always satisfied. In the compact type
case, only the balancing condition is required for the definition of
a twist.

\subsection{Twisted canonical divisors}
\label{tcd}
Let $[C, p_1,\ldots, p_n]\in \mm_{g,n}$, and let
$$I : \widetilde{\mathsf{BN}}(C) \rightarrow \mathbb{Z}\, $$
be a twist.
Let
$\mathsf{N}_I\subset \mathsf{BN}(C)$ be the set of basic nodes
at which $I$ is non-zero,
$$\mathsf{N}_I= \big\{ \, q \in \mathsf{BN}(C) \, \big| \,
I(q,D) \neq 0 \ \ \text{for} \ q\in D \, \big\}\, .$$
Associated to $I$ is the partial normalization
             $$\nu:C_I\rightarrow C$$
defined by normalizing exactly the nodes in $\mathsf{N}_I$.
The curve $C_I$ may be disconnected.

 For a node $q\in \mathsf{N}_I$ in the intersection
of distinct components $D'$ and $D''$ of $C$, we have
$\nu^{-1}(q)=\{q', q''\}$. Let
$$D_q'\subset\nu^{-1}(D') \ \ \text{and} \ \ D_q''\subset\nu^{-1}(D'')$$
denote
the irreducible components of $C_I$ such that $q'\in D_q'$ and $q''\in D_q''$.
By the definition of $\nu$ and the sign condition,
$$D_q'\cap D_q''=\emptyset\, \ \text{in} \ C_I\, .$$


Let $\mu=(m_1,\ldots,m_{n})$  be a vector satisfying
$\sum_{i=1}^n m_i=2g-2$.
To the stable curve $[C,p_1,\ldots,p_n]$, we associate the Cartier divisor
$\sum_{i=1}^n m_i p_i$ on $C$.

\vspace{5pt}
\begin{definition} \label{ffg} The divisor $\sum_{i=1}^n m_i p_i$
associated to $[C,p_1,\ldots,p_n]$ is
{\em twisted canonical} if there exists a twist
$I$ for which
$$\nu^*\OO_{C}\left(\sum_{i=1}^n m_ip_i\right)\, \stackrel{\sim}{=}\,
\nu^*(\omega_C)\otimes \OO_{C_I}
\left(\,\sum_{q\in \mathsf{N}_I} I(q,D'_q)\cdot q'+I(q, D''_q)\cdot q''\right)\,  $$
on the partial normalization $C_I$.
\end{definition}
\vspace{5pt}

We define the subset $\widetilde{\H}_g(\mu)\subset \mm_{g,n}$
parameterizing twisted canonical divisors by
$$\widetilde{\H}_g(\mu)=\left\{[C, p_1, \ldots, p_n]\in \mm_{g,n}\,
\Big| \,  \sum_{i=1}^n m_ip_i \ \ \text{is a twisted canonical divisor}\
\right\}\, .$$
By definition, we have
$$\widetilde{\H}_g(\mu)\cap \cM_{g,n} =\H_g(\mu)\, ,$$
 so $\H_g(\mu) \subset \widetilde{\H}_g(\mu)$ is an
open set.

If $\mu$ has a part equal to $0$, we write $\mu=(\mu',0)$.
Let
$$\tau: \overline{\mathcal{M}}_{g,n} \rightarrow \overline{\mathcal{M}}_{g,n-1}$$
be the map forgetting the $0$ part (when permitted by stability).
As a straightforward consequence of Definition \ref{ffg}, we obtain
$$\widetilde{\H}_g(\mu)=\tau^{-1}\Big(\widetilde{\H}_g(\mu')\Big)\, .$$

\begin{theorem}\label{main1}
If all parts of $\mu$ are non-negative,
$\widetilde{\H}_g(\mu)\subset \mm_{g,n}$ is a closed substack
with irreducible components of dimension either
$2g-2+n$ or $2g-3+n$.
The substack
$$\overline{\H}_g(\mu)\subset \widetilde{\H}_g(\mu)\, $$
is the union of the  components of dimension $2g-2+n$.
%
\end{theorem}

\begin{theorem}\label{main2}
If $\mu$ has a negative part,
$\widetilde{\H}_g(\mu)\subset \mm_{g,n}$ is a closed substack of
pure dimension $2g-3+n$ which contains
$$\overline{\H}_g(\mu)\subset \widetilde{\H}_g(\mu)\, $$
as a union of  components.
\end{theorem}

Theorems \ref{main1} and \ref{main2}
constrain the closures of the
strata of holomorphic and meromorphic differentinals
in geometric and combinatorial terms depending only on the partial normalization $C_I$ of $C$.
By degree considerations on the curve
$[C,p_1,\ldots,p_n]$,
there are only finitely many
twists $I$ which are relevant to determining
whether $\sum_{i=1}^n m_ip_i$ is a twisted canonical divisor.{\footnote{A
proof is given in Lemma 10 of Section \ref{tcb} below.}}

Because of the virtual components in the boundary, Theorems
\ref{main1} and \ref{main2} do not characterize the closure
$\overline{\H}_g(\mu)$. Further constraints can easily
be found on the closure via residue conditions. Our perspective
is not to exclude the virtual components, but rather to include them.
In the strictly meromorphic case, the virtual components
in the boundary are described in Lemma \ref{genn} via star
graphs defined in Section \ref{stt}.
The virtual components may be seen as
a shadow of the virtual fundamental class
of an approach to
the moduli space of differentials via relative stable maps
proposed by  Janda \cite{J}.

In Section \ref{tthh}, we discuss the relationship
between twisted canonical divisors and spin curves via
theta characteristics in case all the parts of $\mu$ are
positive and even.
The basic criterion of Proposition \ref{sqqr2} of Section \ref{xxddx}
 for the smoothability of a
twisted canonical divisor in the holomorphic case is
used in the discussion of the classical examples.

The zero locus of a nontrivial
section of $\omega_C^k$ on a nonsingular curve $C$ is a
{\em $k$-canonical divisor}.
The moduli space
of $k$-canonical divisors is an open set of
the proper moduli space of {\em twisted $k$-canonical divisors}
defined in Section \ref{nndd} for all $k\geq 0$.
However, our results for $k\neq 1$ are weaker
since the associated dimension theory in the case of
 nonsingular curves is yet to be
developed.




\subsection{Related work}
There are several approaches to compactifying the spaces of
canonical divisors:
\begin{enumerate}
\item[$\bullet$] Janda's approach (which has been discussed
briefly above)
is motivated by relative Gromov-Witten theory
and the proof of Pixton's conjecture for the double ramification
cycle \cite{jppz}. Our moduli of twisted canonical divisors
 predicts the image of
Janda's space under the map to $\overline{\mathcal{M}}_{g,n}$. Because of the
pure dimension result of Theorem \ref{main2}
 in the strictly meromorphic case,
the push-forward of the virtual class can also be predicted
(see the Appendix).

The purity of dimension in the strictly meromorphic case allows
for the approach to Pixton's formula taken in
the Appendix based entirely on classical geometry.
The development is unexpected.
For example,
 both the  definition and the calculation
of the double ramification
cycle in \cite{jppz} make essential use of
obstruction theories and virtual fundamental classes.

\vspace{8pt}
\item[$\bullet$] Sauvaget and  Zvonkine \cite{SZ} propose
a different approach to the compactification of canonical divisors
which stays closer to the projectivization of the Hodge bundle
$$ \mathbb{P}(\mathbb{E}) \rightarrow \overline{\cM}_g\, .$$
An advantage is the existence of the tautological line
$\mathcal{O}(1)$
of the projectivization which is hoped eventually to provide a link
to the volume calculations of \cite{EMZ, EO}. There should also be
connections between the recursions for fundamental classes
found by Sauvaget and Zvonkine and the conjecture of the Appendix.

\vspace{8pt}
\item[$\bullet$]
Twisting related (but not equal) to Definition 1
was studied by Gendron \cite{G}.
Chen \cite{Ch}
considered twists in the compact type case
motivated by the study of limit linear series (however
\cite[Proposition 4.23]{Ch}
 for stable curves is related to the theory of twists
presented here).
In further work of Bainbridge, Chen, Gendron, Grushevsky, and M\"oller \cite{BCG},
the authors study twists and residue conditions for stable curves
via analytic methods with the goal of characterizing the
closure $\overline{\H}_g(\mu)$.
\end{enumerate}

\noindent For the study of the moduli of canonical divisors from the point
of Teichm\"uller dynamics, we refer the reader to \cite{EMZ}.

Our paper takes a more naive perspective than the work discussed
above. We propose a precise definition of a twisted canonical divisor
and take the full associated moduli space
seriously as a mathematical object.

\subsection {Acknowledgements}
We thank E. Clader, A. Eskin,  C. Faber,
J. Gu\'er\'e,  F. Jan\-da, A. Pixton, A. Polishchuk,
A. Sauvaget, R. Thomas,
and D. Zvonkine for helpful discussions and correspondence concerning
the moduli space of canonical divisors.
At the
{\em Mathematische Arbeitstagung} in Bonn in June 2015,  there were
several talks and hallway discussions
concerning flat surfaces and the
moduli of holomorphic differentials (involving
D. Chen, S. Grushevsky, M. M\"oller, A. Zorich and others) which
inspired us to write our perspective which had been
circulating as notes for a few years. The precise
connection to Pixton's cycle in the Appendix  was found
in the summer of 2015.

G.F. was supported by the DFG Sonderforschungsbereich {\em Raum-Zeit-Materie}.
R.P. was supported by the Swiss National Science Foundation and
the European Research Council through
grants SNF-200021-143274 and ERC-2012-AdG-320368-MCSK.
R.P was also supported by SwissMap and the Einstein Stiftung.
We are particularly grateful to the Einstein Stiftung for
supporting our collaboration in Berlin.

\section{Twists of degenerating canonical bundles}
\label{degg}
\subsection{Valuative criterion}\label{jjjj}
Let $\mu=(m_1,\ldots,m_n)$ be a vector of zero and pole
multiplicities satisfying
$$\sum_{i=1}^n m_i=2g-2\, .$$
By Definition \ref{ffg},
$$\widetilde{\H}_g(\mu)\subset \mm_{g,n}$$ is easily
 seen to be a construcible subset.

\begin{proposition} \label{xzs}
The locus
$\widetilde{\H}_g(\mu)\subset \mm_{g,n}$ is Zariski closed.
\end{proposition}

To prove Proposition \ref{xzs}, we will use the valuative criterion.
Consider a 1-parameter family of stable $n$-pointed curves over a disk
$\Delta$,
$$\pi:\mathcal{C} \rightarrow \Delta\, , \ \ \ \ \ \ \ \
p_1,\ldots, p_n : \Delta \rightarrow \mathcal{C}\, ,$$
where the
punctured disk $\Delta^\star = \Delta \setminus 0$
maps to  $\widetilde{\H}_g(\mu)$.
The sections $p_i$ correspond to the markings.
We must show
the special fiber over $0\in \Delta$,
$$[C_0,p_1,\ldots,p_n]\, ,$$  also
lies in $\widetilde{\H}_g(\mu)$.

After possibly shrinking $\Delta$,
the topological type of the fibers over $\Delta^\star$
may be assumed to be constant.
After base change, we may assume there is no
monodromy in the components of the fibers
over $\Delta^\star$.
Finally, after further shrinking, we may assume the
twist $I$ guaranteed by Definition 1 for
each fiber over $\Delta^\star$ is the {\em same}.

Since the topological type and the twist $I$ is
the same over $\Delta^\star$, the structures
$$\text{Irr}(C_\zeta)\, ,  \ \ \  \text{Irr}(C_\zeta)^\sim\, , \ \ \
\Gamma_I(C_\zeta)\, $$
do {\em not} vary as the fiber $C_\zeta$ varies
over $\zeta\in\Delta^\star$.

The basic nodes of the special fiber $C_0$ are of two types:
basic nodes smoothed by the family $\pi$
and basic nodes {\em not} smoothed by the family.
The unsmoothed basic nodes correspond to basic nodes of $C_{\zeta\neq 0}$,
so the twist $I$ already assigns integers to the unsmoothed
basic nodes of $C_0$. However, we must assign twists to the
basic nodes of $C_0$ which are smoothed by $\pi$.

For $\zeta\in \Delta^\star$,
consider a vertex $v\in\Gamma_I(C_{\zeta})$ which corresponds
(by taking the union of the irreducible components in $\sim$-equivalence
class $v$)
to a connected subcurve $C^v_\zeta$.
As $\zeta$ varies, an associated family
$$\mathcal{C}^{v} \rightarrow \Delta^\star$$ is defined
with closure in $\mathcal{C}$ given by
$$\pi^v:\overline{\mathcal{C}}^v \rightarrow \Delta\, .$$
The markings which lie on $\overline{\mathcal{C}}^v$ yield sections
$$p_{1_v},\ldots, p_{x_v}: \Delta \rightarrow \overline{\mathcal{C}}^v\, .$$
The nodes connecting $C_\zeta^v$ to the complement in $C_\zeta$
yield further sections
$$q_{1_v},\ldots, q_{y_v}: \Delta \rightarrow \overline{\mathcal{C}}^v\, $$
at which the twist $I$ is {\em not} zero.

By the definition of a twisted canonical divisor, we have
$$\omega_{C_\zeta^v} \stackrel{\sim}= \OO_{C_\zeta^v}\Big(\sum_{i=1}^x m_{i_v}
p_{i_v} -
\sum_{j=1}^y (I(q_{j_v},v)+1)\, q_{j,v}\Big)\, . $$
Hence, the curve
$$[C_\zeta^v, p_{1_v}, \ldots, p_{x_v}, q_{1_v}, \ldots, q_{y_v}]$$
is a twisted canonical divisor for the vector
$$(m_{1_v},\ldots,m_{x_v}, -I(q_{1_v},v)-1,\ldots, -I(q_{y_v},v)-1)\ .$$
In Section \ref{gentw} below, we will show the fiber of
$$\pi^v:\overline{\mathcal{C}}^v \rightarrow \Delta\, , \ \ \ \ \ \ \ \
p_{1_v}, \ldots, p_{x_v}, q_{1_v}, \ldots, q_{y_v}
 : \Delta \rightarrow \mathcal{C}\, ,$$
over $0\in \Delta$ is a twisted canonical divisor (with
nonzero twists only at basic nodes of  $\overline{C}^v_0$
which are smoothed by the family $\pi^v$).

By considering all the vertices $v \in \Gamma_I(C_{\zeta\neq 0})$,
we define twists at all basic nodes of the special
fiber $C_0$ which are smoothed by $\pi$.
We now have defined twists at {\em all} basic nodes of $C_0$,
$$I_0: \widetilde{\mathsf{BN}}(C_0) \rightarrow \mathbb{Z}\, .$$
Via these twists, we easily see
$$[C_0, p_1,\ldots, p_n]$$
is a twisted canonical divisor.
Checking
the balancing, vanishing, sign, and transitivity
conditions is straightforward:
\begin{enumerate}
\item[$\bullet$]Balancing holds for $I_0$ by construction.
\item[$\bullet$]The vanishing, sign, and transitivity for $I_0$ are
all implied by the respective conditions for $I$
and for the twists constructed on
$\overline{\mathcal{C}}^v$.
\item[$\bullet$]
The required isomorphism of line bundles on the partial normalization
$$\nu:C_{0,I_0}\rightarrow C_0$$
determined by $I_0$ is a consequence of
corresponding isomorphisms on the partial normalizations
of $\overline{\mathcal{C}}_0^v$ determined by the twists.
\end{enumerate}

\subsection{Generically untwisted families}
\label{gentw}

In Section \ref{jjjj}, we have reduced the analysis of the
valuative criterion to the generically untwisted case.

Let $\mu=(m_1,\ldots,m_n)$ be a vector of zero and pole
multiplicities satisfying
$\sum_{i=1}^n m_i=2g-2$.
Let
$$\pi:\mathcal{C} \rightarrow \Delta\, , \ \ \ \ \ \ \ \
p_1,\ldots, p_n : \Delta \rightarrow \mathcal{C}\, $$
be a 1-parameter family of stable $n$-pointed curves
for which we have an isomorphism
\begin{equation}
\label{ddff}
\omega_{C_\zeta} \stackrel{\sim}= \OO_{C_\zeta}\Big(\sum_{i=1}^n m_{i}
p_{i}\Big)\,
\end{equation}
for  all $\zeta\in \Delta^\star$.
We will show
the special fiber over $0\in \Delta$,
$$[C_0,p_1,\ldots,p_n]\, ,$$
is a twisted canonical divisor with respect to $\mu$
via a twist
$$I_0: \widetilde{\mathsf{BN}}(C_0) \rightarrow \mathbb{Z}\, $$
supported only on the basic nodes of $C_0$ which are
smoothed by the family $\pi$.

The total space of the family
$\mathcal{C}$ has (at worst) $A_r$-singularities at
the nodes of $C_0$ which are smoothed by $\pi$.
Let
$$\widetilde{\mathcal{C}}
\stackrel{\epsilon}{\longrightarrow} \mathcal{C}
\stackrel{\pi}{\longrightarrow} \Delta$$
be the crepant resolution (via chains of $(-2)$-curves)
of all singularities
occurring at the smoothed nodes of $C_0$.
The  line bundles
\begin{equation}\label{ttww}
\epsilon^* \omega_{\pi} \ \ \ \text{and} \ \ \
\epsilon^* \OO_{\mathcal{C}}\Big(\sum_{i=1}^n m_{i}
p_{i}\Big)\,
\end{equation}
on $\widetilde{\mathcal{C}}$
are isomorphic over $\Delta^\star$ by \eqref{ddff}.
Therefore, the bundles \eqref{ttww}
 differ by a Cartier divisor $\OO_{\widetilde{\mathcal{C}}}(T)$
on
$\widetilde{\mathcal{C}}$ satisfying the following properties:
\begin{enumerate}
\item[(i)] $\OO_{\widetilde{\mathcal{C}}}(T)$ is
supported over $0\in \Delta$,
\item[(ii)] $\OO_{\widetilde{\mathcal{C}}}(T)$
restricts to the trivial
line bundle on every exceptional $(-2)$-curve of the resolution
$\epsilon$.
\end{enumerate}

By property (i), the Cartier divisor $\OO_{\widetilde{\mathcal{C}}}(T)$
 must be a sum of irreducible components of
the fiber
$\widetilde{C}_0$ of $\widetilde{\mathcal{C}}$ over $0\in \Delta$, that is,
$$T = \sum_{D\in \mathsf{Irr}(\widetilde{C}_0)} \gamma_D \cdot D \,, \ \ \mbox{ with }
\gamma_D \in \mathbb{Z} . $$
The irreducible components of $\widetilde{C}_0$
correspond{\footnote{The irreducible components
of $C_0$ maybe be partially normalized in $\widetilde{C}_0$.
We denote the strict transform of
 $D\in \mathsf{Irr}(C_0)$ also by
$D\in \mathsf{Irr}(\widetilde{C}_0)$.}} to the irreducible components of $C_0$
together with all the exceptional $(-2)$-curves,
$$\mathsf{Irr}(\widetilde{C}_0) \ = \ \mathsf{Irr}(C_0) \cup \{E_i\}\ .$$
The Cartier property
implies that
$$\gamma_D = \gamma_{D'}$$
for distinct components $D,D'\in \mathsf{Irr}(C_0)$
which intersect in at least 1 node of $C_0$
 which is {\em not} smoothed
by $\pi$.

Let $E_1 \cup \ldots \cup E_{r}\subset \widetilde{C}_0$ be the full exceptional chain of
$(-2)$-curves for the resolution of an $A_r$-singularity of $\mathcal{C}$
corresponding to a node $q\in C_0$ which is smoothed by $\pi$.
We have the following data:
\begin{enumerate}
\item[$\bullet$]
$E_1$ intersects the irreducible component $D\subset \widetilde{C}_0$,
\item[$\bullet$]
$E_r$ intersects the irreducible component $D'\subset \widetilde{C}_0$,
\end{enumerate}
see Figure 1.
\begin{figure}[h]
 \centering
  \includegraphics[scale=0.3]{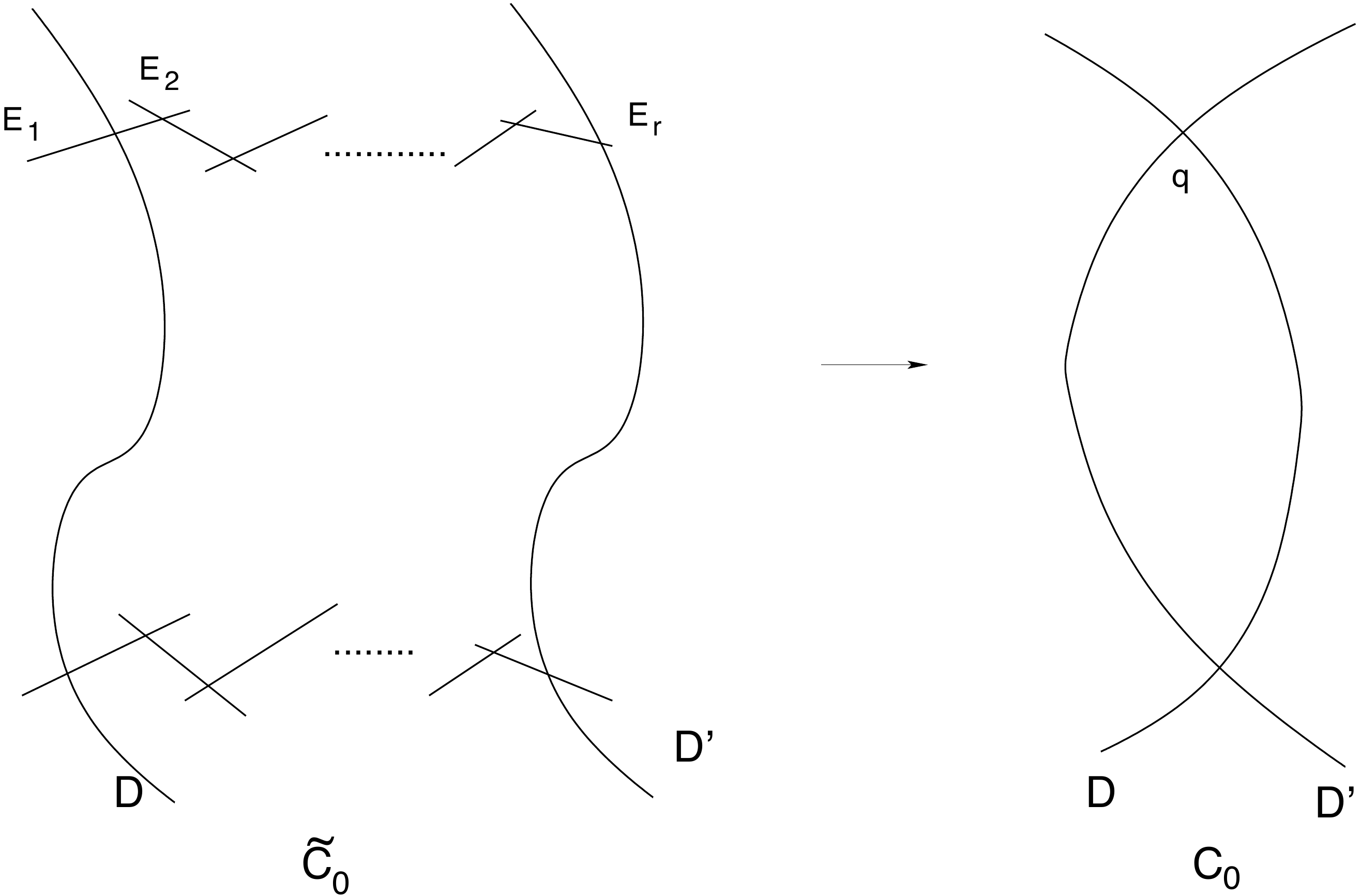}
  \caption{The curve $\widetilde{C}_0$}
\end{figure}
By property (ii), we find $r$ equations obtained
from the triviality of $\OO_{\widetilde{\mathcal{C}}}(T)$ on
each subcurve $E_i$:
\begin{equation}\label{xxqq}
2\gamma_{E_1}=\gamma_D+\gamma_{E_2}\, ,\ \ \ldots \, , \ \
2\gamma_{E_{i}}=\gamma_{E_{i-1}}+\gamma_{E_{i+1}}\,
, \ \ \ldots \, , \ \ 2\gamma_{E_r}=\gamma_{E_{r-1}}+ \gamma_{D'}\, .
\end{equation}
The equations \eqref{xxqq} are uniquely solvable{\footnote{The unique
solution yields $\gamma_{E_i}\in \mathbb{Q}$. For solution
$\gamma_{E_i}\in \mathbb{Z}$, $r+1$ must divide
$-\gamma_{D}+\gamma_{D'}$.}}
in the variables $\gamma_{E_1}, \ldots, \gamma_{E_r}$ in terms of $\gamma_D$ and
$\gamma_{D'}$.
If $\gamma_D=\gamma_{D'}$, then the unique solution is
\begin{equation}\label{xwwx}
\gamma_{E_i}=\gamma_D\, \ \ \text{for all}\ i .
\end{equation}
The solution \eqref{xwwx} is always the case when $q$ is a non-basic node
(since then $D=D'$).

If $q$ is a basic node and $\gamma_D\neq \gamma_{D'}$, then
equations \eqref{xxqq} imply the values of $\gamma_{E_i}$
are uniformly spaced and lie strictly between $\gamma_D$ and $\gamma_{D'}$.
For example if $\gamma_D$=3 and $\gamma_{D'}=9$ and $r=2$, we have
$$3 \ < \ \gamma_{E_1}=5\  <\  \gamma_{E_2}=7\ <\ 9\ . $$
The equations imply
\begin{equation}\label{gqqg}
-\gamma_D+ \gamma_{E_1} = -\gamma_{E_r}+\gamma_{D'}\, .
\end{equation}
If $q\in \mathsf{BN}(C_0)$ is a basic node, then we assign the twists
$$I_0(q,D)= -\gamma_D+ \gamma_{E_1}\, , \ \ \ I_0(q,D')= -\gamma_{D'} + \gamma_{E_r}\ .$$
If there are no $(-2)$-curves over $q\in C_0$, then
$$I_0(q,D)= -\gamma_D+ \gamma_{D'}\, , \ \ \ I_0(q,D')= -\gamma_{D'} +
\gamma_{D}\ .$$
\begin{lemma} The assignment $I_0:\mathsf{BN}(C_0)\rightarrow \mathbb{Z}$
defined above
satisfies the balancing, vanishing, sign, and transitivity
conditions.
\end{lemma}

\begin{proof} As usual, the balancing condition holds by construction
\eqref{gqqg}. We index the components $D\in\mathsf{Irr}(C_0)$
by the integer $\gamma_D$,
$$\gamma : \mathsf{Irr}(C_0) \rightarrow \mathbb{Z}\, , \ \ \ D\mapsto \gamma_D .$$
 The equivalence relation $\sim$ defined by $I_0$
is easily understood: the equivalence classes are exactly
maximal connected subcurves of $C_0$ for which the value of
$\gamma$ is constant. The basic nodes of $C_0$ lying
entirely within such a subcurve have twist $0$ by \eqref{xwwx}.
The basic nodes of $C_0$ connecting different equivalence
classes have non-zero twist by the uniform spacing property
of the solutions to \eqref{xxqq}. The vanishing condition is
therefore established. The sign condition follows again
from the uniform spacing property. A directed
edge in $\mathsf{Irr}(C_0)^\sim$ points in the direction of
higher $\gamma$ values, so transitivity is immediate.
\end{proof}

Finally, in order to show that the special fiber of $\pi$,
$$[C_0,p_1,\ldots,p_n]\, ,$$
is a twisted canonical divisor (with respect to $\mu$ and $I_0$),
we must check the isomorphism
\begin{equation}\label{gg34}
\nu^*\OO_{C_0}\left(\sum_{i=1}^n m_ip_i\right)\, \stackrel{\sim}{=}\,
\nu^*(\omega_{C_0})\otimes \OO_{C_{0,I_0}}
\left(\,\sum_{q\in \mathsf{N}_I} I(q,D'_q)\cdot q'+I(q, D''_q)\cdot q''\right)\,
\end{equation}
on the partial normalization
$$C_{0,I_0}\rightarrow C_0\, .$$

A connected component $C^v_{0,I_0}$ of the partial normalization $C_{0,I_0}$
corresponds to a vertex  $v\in\mathsf{Irr}(C_0)^\sim$.
A node of $C^v_{0,I_0}$, necessarily untwisted by $I_0$, corresponds to a
chain of $(-2)$-curves of $\widetilde{\mathcal{C}}_0$. After inserting
these corresponding chains of rational components at the nodes
of $C^v_{0,I_0}$, we obtain a curve $\widetilde{C}^v_{0,I_0}$ with canonical maps
\begin{equation}\label{llzz7}
 C^v_{0,I_0} \ \leftarrow\ \widetilde{C}^v_{0,I_0}\ \rightarrow\ \widetilde{\mathcal{C}}_0\, .
\end{equation}
The left map of \eqref{llzz7}  contracts the added rational components and
the right map of \eqref{llzz7} is the inclusion.
The isomorphism \eqref{gg34} on $C^v_{0,I_0}$
follows directly from the isomorphism
$$
\epsilon^* \OO_{\mathcal{C}}\Big(\sum_{i=1}^n m_{i}
p_{i}\Big) \, \stackrel{\sim}{=} \,
\epsilon^* \omega_{\pi} \otimes \OO_{\widetilde{\mathcal{C}}}( T)$$
on $\widetilde{C}$ after pull-back via the right map of \eqref{llzz7}
and push-forward via the left map of \eqref{llzz7}.
The proof of Proposition \ref{xzs} is complete. \qed

\subsection{Smoothings} \label{smm}
Let $\mu=(m_1,\ldots,m_n)$ be a vector of zero and pole
multiplicities satisfying
$\sum_{i=1}^n m_i=2g-2$. Let
$$[C,p_1,\ldots,p_n]\in \widetilde{\H}_g(\mu) \subset \mm_{g,n}\, ,$$
be a twisted canonical divisor
via a twist
$$I: \widetilde{\mathsf{BN}}(C) \rightarrow \mathbb{Z}\, .$$

\begin{lemma} \label{gcct} There exists a 1-parameter family
$$\pi:\mathcal{C} \rightarrow \Delta\, , \ \ \ \ \ \ \ \
p_1,\ldots, p_n : \Delta \rightarrow \mathcal{C}\, $$
of stable $n$-pointed curves
and a line bundle
$$\mathcal{L} \rightarrow \mathcal{C}$$
satisfying the following properties:
\begin{enumerate}
\item [(i)] There is an isomorphism of $n$-pointed curves
$$\pi^{-1}(0) \stackrel{\sim}{=}[C,p_1,\ldots,p_n]\, $$
under which
$$\mathcal{L}_0 \stackrel{\sim}{=} \OO_C\left(\sum_{i=1}^n m_i p_i\right)\, .$$
\item[(ii)] The generic fiber
$$C_{\zeta}=\pi^{-1}(\zeta)\, , \ \ \ {\zeta\in \Delta^\star}$$
is a nonsingular curve and
$$\mathcal{L}_\zeta \stackrel{\sim}{=} \omega_{C_\zeta}\, .$$
\end{enumerate}
\end{lemma}

Lemma \ref{gcct} shows the {\em line bundle} associated
to a twisted canonical divisor can always be smoothed to
a canonical line bundle on a nonsingular curve.

\begin{proof}
The twist $I$ determines connected subcurves ${C}_v\subset C$
associated to equivalence classes $v\in \mathsf{Irr}(C)^\sim$.
Let
$$\pi:\mathcal{C} \rightarrow \Delta\, , \ \ \ \ \ \ \ \
p_1,\ldots, p_n : \Delta \rightarrow \mathcal{C}\, $$
be a smoothing of the special fiber $[C,p_1,\ldots,p_n]$.
We can construct a line bundle
$$\mathcal{L} \stackrel{\sim}{=} \omega_\pi\left(
\sum_{v\in \mathsf{Irr}(C)^\sim} \gamma_v \cdot[C_v]\right)$$
for integers $\gamma_v$.
The restriction of $\mathcal{L}$
determines a twist
$$I^\gamma: \widetilde{\mathsf{BN}}(C) \rightarrow \mathbb{Z}\, $$
by the following rule:
if $q\in C_v\cap C_w$, then
$$I^\gamma(q,C_v)= -\gamma_v+\gamma_w\ .$$
In all other cases, $I^\gamma$ vanishes.
An immediate issue is whether
$$\gamma: \mathsf{Irr}(C)^\sim \rightarrow \mathbb{Z}\, \ \ \ \
v \mapsto \gamma_v$$
can be chosen so
\begin{equation}\label{jjjwww}
I^\gamma=I\ .
\end{equation}

Unfortunately, equality \eqref{jjjwww} may be impossible
to satisfy. A simple obstruction is the following:
if $$q,q'\in \mathsf{BN}(C)$$ both lie in the intersection of
the subcurves $C_v$ and $C_w$, then
\begin{equation}\label{ffff}
I^\gamma(q,C_v) = I^\gamma(q',C_w)\ .
\end{equation}
So if
$I(q,C_v) \neq I(q',C_w)$, then $I^\gamma \neq I$ for all $\gamma$.
In order to satisfy \eqref{jjjwww}, we will destabilize the
special fiber $C$ by adding chains of rational components
at the nodes of $\mathsf{BN}(C)$ at which $I$ is supported.

Let $\Gamma_I(C)$ be the directed graph associated to
$I$ with vertex set $\mathsf{Irr}(C)^\sim$.
For $v\in \mathsf{Irr}(C)^\sim$, we define $\text{depth}(v)$
to be the length of the
longest chain of directed edges which ends in $v$.
By the transitivity condition for the twist $I$,
$\text{depth}(v)$ is finite and non-negative.{\footnote{The depth
of $v$ may be $0$.}}
Let
$$M_I = \prod_{(q,D)\in \widetilde{\mathsf{BN}}(C), \, I(q,D)>0}
I(q,D)\ .$$
We define
 $\gamma_v\in \mathbb{Z}$ by
$$\gamma_v=  {\text{depth}(v)} \cdot M_I\ .$$

Let $q\in \mathsf{BN}(C)$ be a node lying in the intersection
$$q \in C_v \cap C_w\ \ \  \text{with} \ \ \  I(q,C_v)>0\, .$$
Since $\text{depth}(v)< \text{depth}(w)$ and
$$ M_I \, | \, -\gamma_v+\gamma_w  \, \  \ \ \text{we have } \ \ \
I(q,C_v) \, | \, -\gamma_v+\gamma_w\ .$$
We add a chain of
$$\frac{-\gamma_v+\gamma_w}{I(q,C_v)} -1$$
destabilizing rational curve at each such node
$q$. For each
rational curve $P^q_i$ in the chain
$$P^q_1\, \cup\, \ldots\,\cup\, P^q_{\frac{-\gamma_v+\gamma_w}{I(q,C_v)} -1}\, ,$$
we define
$\gamma_{P^q_i} = \gamma_v+ i\cdot {I(q,C_v)}$.

The result is a new curve $\widetilde{C}$ with a
map
$$\widetilde{C} \rightarrow C$$
contracting the added chains. Moreover, for a
1-parameter family
\begin{equation}\label{fff444}
\widetilde{\pi}:\widetilde{\mathcal{C}} \rightarrow \Delta\, , \ \ \ \ \ \ \ \
p_1,\ldots, p_n : \Delta \rightarrow \widetilde{\mathcal{C}}\,
\end{equation}
smoothing the special fiber $[\widetilde{C},p_1,\ldots,p_n]$, we
construct a line bundle{\footnote{The subcurve $C_v\subset C$
corresponds to a subcurve $C_v\subset \widetilde{C}$ by strict
transformation.}}
$$\widetilde{\mathcal{L}} \stackrel{\sim}{=} \omega_{\widetilde{\pi}}\left(
\sum_{v\in \mathsf{Irr}(C)^\sim} \gamma_v \cdot[C_v]
+\sum_{q\in\mathsf{BN}(C)} \sum_{i=1}^{{\frac{\gamma_w-\gamma_v}{I(q,C_v)} -1}}
\gamma_{P_i^q} \cdot [P_i^q]
\right)\ .$$
The line bundle $\widetilde{\mathcal{L}}$ satisfies
several properties:
\begin{enumerate}
\item[(i)] $\widetilde{\mathcal{L}}|_{P_i^q}$ is trivial,
\item[(ii)] for every $v\in \mathsf{Irr}(C)^\sim$,
$$\widetilde{\mathcal{L}}|_{C_v}\, \stackrel{\sim}{=}\,
\omega_C|_{C_v}\otimes \OO_{C_v}
\left(\,\sum_{q\in \mathsf{N}_I} I(q,C_v)\cdot q \right)\,   \stackrel{\sim}{=}\,
\OO_{C_v}\left(\sum_{p_i\in C_v} m_ip_i\right)\, ,$$
\item[(iii)] for $\zeta\in \Delta^\star$,
$\widetilde{\mathcal{L}}_{\zeta} \stackrel{\sim}{=}
\omega_{\widetilde{C}_\zeta}$.
\end{enumerate}
We can contract the extra rational components $P_i^q$
in the special fiber of $\widetilde{\mathcal{C}}$ to obtain
1-parameter family of stable $n$-pointed curves
$$\pi:\mathcal{C} \rightarrow \Delta\, , \ \ \ \ \ \ \ \
p_1,\ldots, p_n : \Delta \rightarrow \mathcal{C} $$
which smooths $[C,p_1,\ldots,p_n]$.
By (i), the line bundle $\widetilde{\mathcal{L}}$
descends to
$$\mathcal{L} \rightarrow \mathcal{C}\ .$$
By (ii), for every $v\in\mathsf{Irr}(C)^\sim$,
\begin{equation}\label{qq22}
\mathcal{L}|_{C_v} \stackrel{\sim}{=}\,
\OO_{C_v}\left(\sum_{p_i\in C_v} m_ip_i\right)\, .
\end{equation}
The isomorphisms \eqref{qq22} after restriction to the subcurves
do {\em not} quite imply the required isomorphism
\begin{equation}\label{tt66}
 \mathcal{L}_0=\mathcal{L}|_{C} \,\stackrel{\sim}{=}\,
\OO_{C}\left(\sum_{i=1}^n m_ip_i\right)
 \end{equation}
because there are additional $h^1(\Gamma_I)$ factors of  $\mathbb{C}^*$
in the Picard variety of $C$. However, these $\mathbb{C}^*$-factors for
$$\mathcal{L}|_{C}\ \ \ \text{and}\ \ \ \OO_{C_v}\left(\sum_{i=1}^n m_ip_i\right)$$
can be matched by correctly choosing the smoothing parameters at the
nodes of $\widetilde{C}$ in
the original family \eqref{fff444}.
\end{proof}

\section{Dimension estimates}
\subsection{Estimates from above} \label{efa}
Let $\mu=(m_1,\ldots,m_n)$ be a vector{\footnote{For the
dimension estimates here, we consider all $\mu$. No assumptions
on the parts are made.}} of zero and pole
multiplicities satisfying
$\sum_{i=1}^n m_i=2g-2$.
The {\em boundary}
$$\partial \mm_{g,n} \subset \mm_{g,n}$$
is the divisor parameterizing stable $n$-pointed curves
with at least one node.
We estimate from above the dimension of irreducible components
of $\widetilde{\H}_g(\mu)$ supported in the boundary.

\begin{proposition}  \label{dim1}
Every irreducible
component of $\widetilde{\H}_g(\mu)$ supported entirely
in the boundary of $\mm_{g,n}$ has dimension at most
$2g-3+n$.
\end{proposition}

\begin{proof}
Let $Z\subset \widetilde{\H}_g(\mu)$ be an irreducible component
supported entirely in $\partial \mm_{g,n}$.
Let
$\Gamma_{Z}$
 be the dual graph of $C$ for
the generic element
$$[C,p_1,\ldots,p_n] \in Z\, . $$
The dual graph{\footnote{A refined
discussion of dual graphs via half-edges is required in the
Appendix. Here, a simpler treatment is sufficient.}} consists of vertices (with genus assignment), edges, and legs
(corresponding to the markings):
$$\Gamma_{Z}= (\mathsf{V},\, \mathsf{E}\, , \mathsf{L}\,, \g:\mathsf{V}\rightarrow \mathbb{Z}_{\geq 0})\ . \ $$
Each vertex $v\in \mathsf{V}$ corresponds to an
irreducible component of $C$.
The valence $\n(v)$ counts both half-edges and legs
incident to $v$.
The genus formula is
$$g-1 = \sum_{v\in \mathsf{V}} (\g(v)-1) + |\mathsf{E}|\, .$$
Since $Z$ is supported in the boundary, $\Gamma_{Z}$
must have at least one edge.


We estimate from above the dimension of the moduli of twisted canonical
divisors in $\mm_{g,n}$ with dual graph exactly $\Gamma_{Z}$.
Let $v\in \mathsf{V}$ be a vertex corresponding
to the moduli space $\mathcal{M}_{\g(v),\n(v)}$.
The dimension of the moduli of the canonical divisors {\em on the
the component corresponding to $v$} is
bounded from above by
\begin{equation} 2\g(v)-2+ \n(v)\, .
\label{xxqq22}
\end{equation}
Here, we have used the dimension formula for the locus of
canonical divisors with prescribed zero
multiplicities \cite{Pol}. The strictly meromorphic case
has lower dimension (see Section \ref{zp}).

Summing over the vertices yields a dimension bound for $Z$:
\begin{eqnarray}
\label{ffrrdd} \mbox{dim }  Z & \leq&
 \sum_{v\in \mathsf{V}} 2\g(v)-2 +\n(v) \\ \nonumber
 & = & 2\sum_{v\in \mathsf{V}} (\g(v)-1) + 2|\mathsf{E}|+ n \\  \nonumber
& = & 2g-2 + n\, . \nonumber
\end{eqnarray}
The result falls short of the claim of the Proposition by 1.

We  improve the bound \eqref{ffrrdd} by the following
observation. The dual graph $\Gamma_Z$ must have at least one edge $e$
since $Z$ is supported in the boundary:
\begin{enumerate}
\item[$\bullet$]
If the edge $e$ corresponds to a non-basic node, then the vertex
incident to $e$ is allowed a pole at the two associated nodal points.
\item[$\bullet$]
If the edge $e$ corresponds to a basic node,
 there must be at least one vertex
$v\in \mathsf{V}$ incident to $e$ which carries meromorphic differentials.
\end{enumerate}
In either case,
we can apply the stronger bound
\begin{equation} 2\g(v)-3+ \n(v)\ .
\label{xxqq223}
\end{equation}
Since we are guaranteed at least one application of
\eqref{xxqq223}, we conclude
$$ \mbox{dim } Z \leq 2g-3+n$$
which is better than \eqref{ffrrdd} by 1.
\end{proof}

\subsection{Star graphs} \label{stt}
Let $Z\subset \widetilde{\H}_g(\mu)$ be an irreducible component
supported entirely in $\partial \mm_{g,n}$.
By the proof of Proposition \ref{dim1}, if
$\Gamma_Z$
has for every twist $I$ at least {\em two} vertices corresponding to meromorphic
differentials, then
$$ \mbox{dim } Z \leq 2g-4+n\, .$$
The dual graphs $\Gamma_Z$ with a least one edge (since
$Z$ is supported in the boundary) and carrying a twist $I$
with only one vertex
corresponding to meromorphic differentials are easy to classify.
The dual graph $\Gamma_Z$ must be a {\em star}:
\begin{enumerate}
\item[$\bullet$] The vertices $\{v_0,v_1,\ldots, v_k\}$ of $\Gamma_Z$
consist of a {\em center} vertex $v_0$ and {\em outlying} vertices
$\{v_1,\ldots, v_k\}$.
\item[$\bullet$] The edges of $\Gamma_Z$ are of two types:
\begin{enumerate}
\item[(i)] self-edges at $v_0$,
\item[(ii)] edges (possibly multiple) connecting the center
$v_0$ to the
outlying vertices.
\end{enumerate}
\item[$\bullet$] The parts{\footnote{The parts of $\mu$
correspond to the markings.}} of $\mu$
are distributed to the
vertices with {\em all}
negative parts   distributed to the center $v_0$.
\end{enumerate}

\vspace{5pt}
\begin{centering}{\hspace{50pt} \includegraphics[scale=0.5]{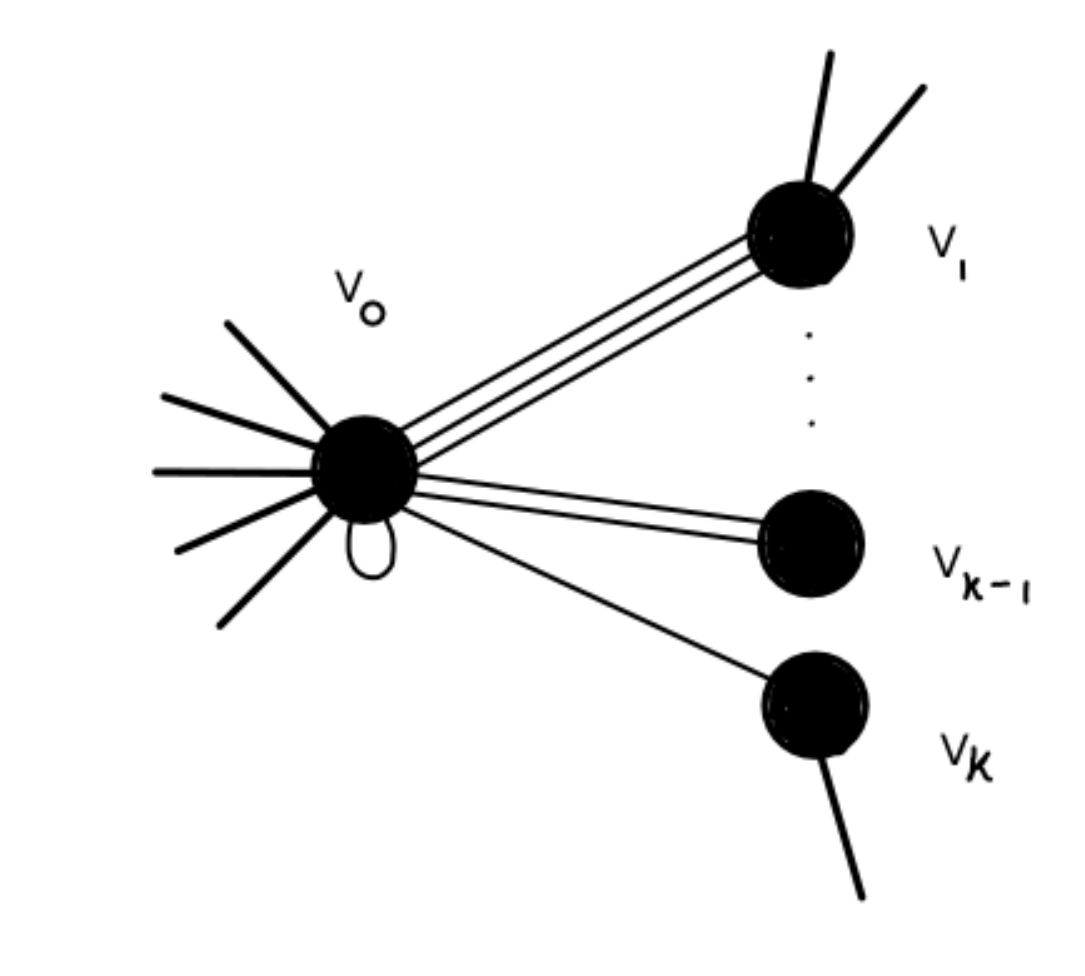}}
\end{centering}

\noindent A twist $I$ which orients every edge of type (ii)
 as outgoing from $v_0$
yields meromorphic differentials at $v_0$ only. The self-edges
at $v_0$ correspond to non-basic nodes and can not be twisted.

\subsection{Twisted canonical bundles} \label{tcb}
We have defined a twisted canonical {\em divisor} in Section \ref{tcd}.
A twisted canonical divisor determines a section (up to scale) of a twisted
canonical {\em bundle} on $C$.

\vspace{5pt}
\begin{definition} \label{ffgl} A line bundle $L$ on a connected
nodal curve
$C$ is
{\em twisted canonical} if there exists a twist
$$I: \widetilde{\mathsf{BN}}(C) \rightarrow \mathbb{Z}$$
for which there is an isomorphism
$$\nu^*L\, \stackrel{\sim}{=}\,
\nu^*(\omega_C)\otimes \OO_{C_I}
\left(\,\sum_{q\in \mathsf{N}_I} I(q,D'_q)\cdot q'+I(q, D''_q)\cdot q''\right)\,  $$
on the partial normalization $\mu:C_I\rightarrow C$.
\end{definition}
\vspace{5pt}

The limits of twisted canonical line bundles are also twisted
canonical. The proof has already been given as a step in
the proof of the properness of $\widetilde{\mathcal{H}}_g(\mu)$
in Section \ref{degg}. We state the result as a Lemma.

\begin{lemma}\label{fffff}
Let $\pi: \mathcal{C} \rightarrow \Delta$ be a flat family of
connected nodal curves, and let
$$\mathcal{L} \rightarrow \mathcal{C}$$
be a line bundle.
If $L_\zeta \rightarrow C_\zeta$ is a twisted canonical
line bundle for all $\zeta\in \Delta^\star$, then
$$L_0\rightarrow C_0$$
is also a twisted canonical line bundle.
\end{lemma}

The definition of a twisted canonical line bundle does {\em not}
specify a twist. Only the existence of a twist is required.
However, there are only finitely many possible twists.

\begin{lemma} If $L\rightarrow C$ is a twisted
canonical line bundle, there are only {\em finitely} many twists
$$I: \widetilde{\mathsf{BN}}(C) \rightarrow \mathbb{Z}$$
for which there exists an isomorphism
$$\nu^*L\, \stackrel{\sim}{=}\,
\nu^*(\omega_C)\otimes \OO_{C_I}
\left(\,\sum_{q\in \mathsf{N}_I} I(q,D'_q)\cdot q'+I(q, D''_q)\cdot q''\right)\,  $$
on the partial normalization $\mu:C_I\rightarrow C$.
\end{lemma}

\begin{proof}
Let $I$ be a twist for which the above isomorphism holds.
By definition,  $I$ must satisfy the balancing, vanishing,
sign, and transitivity conditions of Section \ref{twww}.
Since there are no directed loops in the graph $\Gamma_I(C)$,
there must be a vertex $v$ with only outgoing arrows.
Let $$C_v \subset C\ $$
be the associated subcurve.
The twist $I$ is always positive on the nodes
$\mathsf{N}_I$ incident to the subcurve $C_v$ (on the
side of $C_v$).
Since the degrees of $\nu^*(L)|_{C_v}$ and
$\nu^*(\omega_C)|_{C_v}$ are determined, we obtain
a bound on the twists of $I$ on the nodes $\mathsf{N}_I$
incident to $C_v$: only finitely many values for
$I$ on these nodes are permitted.

Next, we find a vertex $v'\in \Gamma_I(C)$ which has only
 outgoing arrows (except from possibly $v$).
Repeating the above argument easily yields the
required finiteness statement for $I$.
\end{proof}

\subsection{Estimates from below}\label{efb}
Let $\mu=(m_1,\ldots,m_n)$ be a vector of zero and pole multiplicities
satisfying
$\sum_{i=1}^n m_i=2g-2$.
We now prove a lower bound on the dimension of irreducible components
of $\widetilde{\H}_g(\mu)$.

\begin{proposition}  \label{dim2}
Every irreducible
component of $\widetilde{\H}_g(\mu)$
has dimension at least
$2g-3+n$.
\end{proposition}

\begin{proof}
Let $[C, p_1, \ldots, p_n]\in \widetilde{\H}_g(\mu)$
be a stable $n$-pointed curve. Let
$$L \stackrel{\sim}{=} \mathcal{O}_C\left(\sum_{i=1}^{n} m_ip_i\right)\ \rightarrow\
C$$
be the associated twisted canonical bundle.
Let
$$m_1,\ldots, m_k$$
be the negative parts of $\mu$ (if there are any).
Consider the $k$-pointed nodal curve
$$[C,p_{1}, \ldots,p_k]$$
obtained by dropping{\footnote{We do {\em not} contract
unstable components.}} the markings $p_{k+1},\ldots,p_n$.
Let $r_1$ and $r_2$ denote the number of rational
components of $[C,p_{1}, \ldots,p_k]$ with exactly
$1$ and $2$ special points respectively.
To kill the automorphisms of these unstable component, we add new markings
$$q_1,\ldots, q_{2r_1+r_2}$$
to the curve $C$ (away from  $p_1,\ldots,p_k$ and
away from the nodes):
\begin{enumerate}
\item[$\bullet$] we add two $q's$ on each  components
                 with 1 special point,
\item[$\bullet$] we add one  $q$ on each  components with 2
special points.
\end{enumerate}
The result
\begin{equation} \label{vvzz}
[C,p_1,\ldots,p_k,q_1,\ldots, q_{2r_1+r_2}]
\end{equation}
is a stable pointed curve.

Let $\mathcal{V}$ be the nonsingular
versal deformation space of the $k+2r_1+r_2$-pointed
curve \eqref{vvzz},
$$\text{dim}\, \mathcal{V} =\text{dim}\, \mbox{Def}([C,p_1\ldots,p_k,
q_1,\ldots, q_{2r_1+r_2}])=3g-3+k+2r_1+r_2.$$
There is a universal curve{\footnote{We will not use the $q$-sections.}}
$$\pi: \mathcal{C} \rightarrow \mathcal{V}\, , \ \ \
p_1,\ldots, p_k: \mathcal{V} \rightarrow \mathcal{C}\, .$$
We consider the relative moduli space{\footnote{The Quot
scheme parameterization of line bundles
can be used to avoid the separation issues of
the moduli of line bundles. The dimension calculus is
parallel.}} of degree $2g-2$
line bundles on the fibers of $\pi$,
 $$\epsilon:\cB\rightarrow \mathcal{V}\, .$$

Let $\mathcal{V}^\star\subset \mathcal{V}$ be the locus of nonsingular curves
in the versal deformation space, and let
$$\cB^\star\rightarrow \mathcal{V}^\star$$ be the relative
Jacobian of degree $2g-2$.
 Let $$\cW^\star \subset \cB^\star$$ be the codimension $g$ locus in the
universal Jacobian defined fiberwise by the relative canonical bundle
$\omega_\pi$.
Let $\cW$  be the closure of $\cW^\star$ in $\cB$.
By Lemma \ref{fffff}, every point of $\cW$ parameterizes a
twisted canonical bundle.
By Lemma \ref{gcct},
the line bundle $L$ lies in $\cW$ over the special
fiber $[C]\in \mathcal{V}$.
The dimension of $\cW$ is
$$\text{dim}\, \cW \, =\,  \text{dim}\, \cB^\star -g\, =\, 3g-3+k+2r_1+r_2\ .$$

Let $[C',L']\in \cW$ be a pair where $[C']\in \mathcal{V}$ and
$$L'\rightarrow C'$$
is a twisted canonical bundle.  A {\em good section} of
$L'(-\sum_{i=1}^km_ip_i)$ on $C'$ is section $s$ satisfying the
following properties:
\begin{enumerate}
\item[$\bullet$] $s$ does {\em not} vanish at $p_1,\ldots,p_k$,
\item[$\bullet$] $s$ does {\em not}
vanish at any node of $C'$,
\item[$\bullet$] $s$ does {\em not} vanish identically
on any irreducible component of $C'$,
\item[$\bullet$] the points $p_1,\ldots,p_k$ and
$\text{Div}(s)$ together stabilize $C'$.
\end{enumerate}
Good sections are open in the space of sections of $[C',L']$.
The zeros of a good section define a twisted
canonical divisor.

Since we started with a twisted canonical divisor,
$[C, p_1, \ldots, p_n]\in \widetilde{\mathcal{H}}_g(\mu)$. We have
a good section
\begin{equation}\label{ffww}
L\left(-\sum_{i=1}^k m_i p_i\right)\, \stackrel{\sim}{=}\,
\mathcal{O}_C\left(\sum_{i=k+1}^n m_ip_i\right)\, .
\end{equation}
 associated to the
pair $[C,L]$.

We now can estimate the dimension of the space of good sections of
$$L'\left(-\sum_{i=1}^k m_i p_i\right)$$
as $[C',L']$ varies in $\cW$
near $[C,L]$ using Proposition \ref{qq1} of Section \ref{qq}.
The local dimension of the space of sections near $[C,L]$
is at least
\begin{eqnarray*}
\text{dim}\, \cW + \chi(C,L) & = &  3g-3+k+2r_1+r_2 + g-1-\sum_{i=1}^k m_i \\
& = &
4g-4+k+2r_1+r_2- \sum_{i=1}^k m_i\, .
\end{eqnarray*}
Hence, dimension of the space $\mathcal{T}$ of twisted
canonical divisors on the fibers of
$$\pi: \mathcal{C} \rightarrow \cW\,$$
near $[C, p_1, \ldots, p_n]$ is {\em at least} $4g-5+k+2r_1+r_2-\sum_{i=1}^k m_i$
.

We must further impose conditions on the twisted
canonical divisors parameterized by $\mathcal{T}$
to obtain the shape $\sum_{i=k+1}^n m_i p_i$ for the positive part.
These conditions impose at most
$$2g-2-\sum_{i=1}^km_i -(n-k)= 2g-2+k -n - \sum_{i=1}^k m_i$$
constraints. Hence, we conclude the dimension of
space $\mathcal{T}(\mu)$ of twisted
canonical divisors on the fibers of
\begin{equation}\label{szzz}
\pi: \mathcal{C} \rightarrow \cW\,
\end{equation}
near $[C, p_1, \ldots, p_n]$ of shape $\mu$ is {\em at least}
\begin{multline*}
4g-5+k+2r_1+r_2-\sum_{i=1}^k m_i - \left(2g-2 + k -n-\sum_{i=1}^k m_i\right)
\\ = 2g-3+n+2r_1+r_2\, .
\end{multline*}

Suppose  $[C, p_1, \ldots, p_n]\in \widetilde{\mathcal{H}}_g(\mu)$
is a generic element of an irreducible component $Z\subset \widetilde{\H}_g(\mu)$
of dimension strictly less than $2g-3+n$.
In the versal deformation space $\mathcal{V}$ above, we consider
the dimension of the sublocus
$$\mathcal{Z} \subset \mathcal{T}(\mu)$$
corresponding to twisted canonical divisors on the
fibers of \eqref{szzz} which have the {\em same}
dual graph as $C$. The dimension of the sublocus $\mathcal{Z}$
is equal to
$$\text{dim}\, Z + 2r_1+r_2 $$
which is less than $2g-3+n+2r_1+r_2$.
  The summands $2r_1$ and $r_2$ appear here since we do not
now quotient by the automorphism of the unstable components
of the fibers.

We conclude at least one node of the curve $C$ can be
smoothed in $\widetilde{\H}_g(\mu)$. Therefore, there
does not exist a component $Z\subset \widetilde{\H}_g(\mu)$
of dimension strictly less than $2g-3+n$.
\end{proof}

\section{Theorems \ref{main1} and \ref{main2}}

\subsection{Proof of Theorem \ref{main1}}
By Proposition \ref{xzs},
$$\widetilde{\H}_g(\mu) \subset \overline{\mathcal{M}}_{g,n}$$
is a closed subvariety.
In case all parts of
$\mu=(m_1,\ldots,m_n)$
 are non-negative,
every irreducible component of $${\H}_g(\mu)\subset \mathcal{M}_{g,n}$$
has dimension $2g-2+n$.
Hence, every irreducible component of
$$\overline{\H}_g(\mu)\subset \overline{\mathcal{M}}_{g,n}$$
has dimension $2g-2+n$.

By Proposition \ref{dim1}, every irreducible component of $\widetilde{\H}_g(\mu)$
 which is supported  in the boundary $\partial\overline{\mathcal{M}}_{g,n}$ has dimension at most $2g-3+n$.
By Proposition \ref{dim2}, every irreducible component of $\widetilde{\H}_g(\mu)$
 has dimension
at least $2g-3+n$.
Hence, the boundary components of ${\H}_g(\mu)$ all have dimension
$2g-3+n$.
\qed

\subsection{Locus of irreducible curves} \label{irrr}
From the point of view of twists, the locus  of stable pointed curves
$$\mathcal{M}_{g,n}^{\mathsf{Irr}} \subset \overline{\mathcal{M}}_{g,n}$$
with irreducible domains is very natural to consider,
$$[C,p_1,\ldots,p_n] \in \mathcal{M}_{g,n}^{\mathsf{Irr}} \ \ \
\leftrightarrow \ \ \ \text{$C$ is irreducible.}$$
Since an irreducible curve has no basic nodes, a twisted canonical divisor
$$[C,p_1,\ldots,p_n] \in \widetilde{\H}_g(\mu) \, \cap\,
\mathcal{M}_{g,n}^{\mathsf{Irr}}$$
is a usual canonical divisor
$$\mathcal{O}_C\left(\sum_{i=1}^n m_ip_i\right) \stackrel{\sim}{=}
\omega_C\ .$$
We define $\H_g^{\mathsf{Irr}}(\mu)$ by the intersection
$$\H_g^{\mathsf{Irr}}(\mu) =
\widetilde{\H}_g(\mu) \cap \mathcal{M}_{g,n}^{\mathsf{Irr}}\, .$$
\begin{lemma} If all parts of $\mu$ are non-negative,
$$\overline{\H}_g(\mu) \cap \mathcal{M}_{g,n}^{\mathsf{Irr}}
\,= \,
{\H}_g^{\mathsf{Irr}}(\mu) \, .
$$ \label{sqqr}
\end{lemma}
\begin{proof}
We must prove there does not exist a component of
$Z\subset \widetilde{\H}_g(\mu)$ generically supported in
the boundary $\partial \mathcal{M}_{g,n}^{\mathsf{Irr}} $.
We have seen that the dimension of $Z$ must be $2g-3+n$.
However, every irreducible component of the locus
of canonical divisors in
$\mathcal{M}_{g,n}^{\mathsf{Irr}}$ has dimension at least
$2g-2+n$ since the canonical bundle always has $g$ sections.
\end{proof}

In the strictly meromorphic case, the equality
\begin{equation*}
\overline{\H}_g(\mu) \cap \mathcal{M}_{g,n}^{\mathsf{Irr}}
\,= \,
{\H}_g^{\mathsf{Irr}}(\mu) \, .
\end{equation*}
also holds, but the result does
not follow from elementary dimension arguments. Instead,
the analytic perspective of flat surface is required, see \cite{Ch,SZ,W}.

By Lemma \ref{sqqr} in the holomorphic case and Theorem \ref{main2}
in the strictly meromorphic case, we conclude the following two
dimension results:
\begin{enumerate}
\item[(i)] If all parts of $\mu$ are non-negative,
${\H}_g^{\mathsf{Irr}}(\mu)$
has pure dimension $2g-2+n$.
\item[(ii)] If $\mu$ has a negative part,
${\H}_g^{\mathsf{Irr}}(\mu)$
has pure dimension $2g-3+n$.
\end{enumerate}
The dimensions (i) and (ii) are identical to the dimensions
of
$\H_g(\mu)$ in the corresponding cases.

Since $\overline{\H}_{g}(\mu) = \overline{\H}_g^{\mathsf{Irr}}(\mu)$ for all $\mu$,
the star graphs of Section \ref{stt} with self-edges at the center vertex
do {\em not} correspond to the dual graph $\Gamma_Z$ of a generic element of an irreducible
component $Z\subset \widetilde{\H}_g(\mu)$ supported in the boundary $\partial \overline{\mathcal{M}}_{g,n}$.
We state the result as a Lemma.

\begin{lemma} \label{genn} The dual graph $\Gamma_Z$ of a generic element of an irreducible
component $$Z\subset \widetilde{\H}_g(\mu)$$ supported in the boundary $\partial \overline{\mathcal{M}}_{g,n}$
 is a star
graph with no self-edges at the center.
\end{lemma}

Such star graphs will be called {\em simple} star graphs.
In the Appendix, in the strictly meromorphic case, we will include in the set of simple star graphs
the trivial star graph with
a center vertex carrying all the parts of $\mu$ and {\em no} edges or
outlying vertices.


\subsection{Proof of Theorem \ref{main2}}
By Proposition \ref{xzs},
$$\widetilde{\H}_g(\mu) \subset \overline{\mathcal{M}}_{g,n}$$
is a closed subvariety.
In case $\mu$ has at least one negative part,
every irreducible component of
$$\H_g(\mu) \subset \mathcal{M}_{g,n}$$
has dimension $2g-3+n$, and every irreducible component
of $\widetilde{\H}_g(\mu)$ which is supported in the boundary
$\partial\overline{\mathcal{M}}_{g,n}$
has dimension at most $2g-3+n$ by Proposition \ref{dim1}.
By Proposition \ref{dim2} every irreducible component of $\widetilde{\H}_g(\mu)$
has dimension
at least $2g-3+n$. Hence, $\widetilde{\H}_g(\mu)$
is pure of dimension $2g-3+n$.
\qed

\section{Sections of line bundles}


\subsection{Dimension estimates}
\label{qq}
Let $\mathcal{X}$ be a variety of pure dimension $d$, and let
$$\pi: \mathcal{C} \rightarrow \mathcal{X}$$
 be a flat family of nodal curves.
Let
$$\mathcal{L} \rightarrow \mathcal{C}$$
 be a line bundle with Euler characteristic $\chi(\mathcal{L})$
on the fibers of $\pi$. Let
$$\phi:\mathcal{M} \rightarrow \mathcal{X}$$
be the variety parameterizing {\em nontrivial} sections of $\mathcal {L}$
on the fibers of $\pi$,
$$\phi^{-1}(\xi) =H^0(\mathcal{C}_\xi,\mathcal{L}_\xi)\setminus \{0\}$$
for $\xi\in \mathcal{X}$.

\begin{proposition}
Every irreducible component of $\mathcal{M}$ has dimension at least
$d + \chi(\mathcal{L})$. \label{qq1}
\end{proposition}



The result is immediate from the point of view of obstruction theory.
For the convenience of the reader, we include
an elementary proof based on the construction of
 the parameter space  $\mathcal{M}$.

\begin{proof}
Using a $\pi$-relatively ample bundle $\mathcal{N} \rightarrow \mathcal{C}$,
we form a quotient
                   $$\oplus_{i=1}^r \mathcal{N}^{-k_i} \rightarrow
\mathcal{L} \rightarrow 0\, , \ \ \ \ \ \  k_i>0.$$
Let $\mathcal{A}=\oplus_{i=1}^r \mathcal{N}^{-k_i}$ and consider
the associated short exact sequence
\begin{equation}\label{vv999}
0 \, \rightarrow\, \mathcal{B} \, \rightarrow\, \mathcal{A}\, \rightarrow\,
\mathcal{L}\, \rightarrow\, 0
\end{equation}
of bundles  on $\mathcal{C}$.
Since $\mathcal{A}$ is $\pi$-relatively negative,
$\mathcal{A}$ and $\mathcal{B}$ have no sections on the fibers of $\pi$. Hence,
$$V_{\mathcal{A}} = R^1\pi_*(\mathcal{A})\, \ \ \ \text{and} \ \ \
V_{\mathcal{B}} = R^1\pi_*(\mathcal{B})$$
are vector bundles of ranks $a$ and $b$ on  $\mathcal{X}$.
For $\xi \in \mathcal{X}$, we have
\begin{equation}\label{z23z}
0\, \rightarrow\, H^0(\mathcal{C}_\xi, \mathcal{L}_\xi)
\, \rightarrow \, V_{\mathcal{B},\xi}\,
\, \rightarrow V_{\mathcal{A},\xi}  \, \rightarrow\,
H^1(\mathcal{C}_\xi, \mathcal{L}_\xi)\, \rightarrow 0\, .
\end{equation}
Therefore, the ranks satisfy
 $b-a=\chi(L)$ .

The sequence in cohomology associated to \eqref{vv999} yields a bundle
map
$$ \gamma:V_{\mathcal{B}} \rightarrow V_{\mathcal{A}}$$
 on $\mathcal{X}$.
Let $V^*_{\mathcal{B}}$ be the total space of $V_{\mathcal{B}}$ with the zero section removed,
$$q: V^*_{\mathcal{B}} \rightarrow \mathcal{X}\, .$$
The pull-back to $V^*_{\mathcal{B}}$ of $V_{\mathcal{B}}$ carries
a tautological line bundle
$$ 0\rightarrow Q \rightarrow q^*(V_{\mathcal{B}})\, .$$
By \eqref{z23z},
the parameter space
$\mathcal{M}$ sits in $V^*_{\mathcal{B}}$ as the zero locus of the canonical map
$$Q \rightarrow  q^*(V_{\mathcal{A}})$$
obtained from the composition
$$Q \, \rightarrow \, q^*(V_{\mathcal{B}})\, \stackrel{q^*(\gamma)}{\longrightarrow}
\, q^*(V_{\mathcal{A}})\, .$$

Since $\mathcal{X}$ is of pure dimension $d$,
$V^*_{\mathcal{B}}$ is of pure dimension $d+b$. Finally, since
the rank of $q^*(V_{\mathcal{A}})$ is $a$, every irreducible component
of $\mathcal{M}$ is of dimension at least
$$d+b-a= d+ \chi(\mathcal{L})\ ,$$
by standard dimension theory.
\end{proof}

The following result is a consequence of the existence
of a reduced obstruction theory (see the cosection method
of \cite{KL}). Again, we give an elementary proof.
An application will appear in Section \ref{xxddx} below.

\begin{proposition}
If there exists a trivial quotient on $\mathcal{X}$,
$$R^1\pi_* \mathcal{L} \rightarrow \XC \rightarrow 0\, ,$$
then every irreducible component of $\mathcal{M}$ has dimension at least
$d + \chi(\mathcal{L})+1$. \label{qq2}
\end{proposition}

\begin{proof}
We must improve the dimension estimate in Proposition \ref{qq1}
by 1.
The long exact sequence obtained from \eqref{vv999} yields a
quotient
$$V_{\mathcal{A}} \rightarrow R^1\pi_*(L) \rightarrow 0$$
on $\mathcal{X}$.  Composing with the given quotient
$$R^1\pi_*(\mathcal{L}) \rightarrow \XC \rightarrow 0$$
yields a quotient
  $$ V_{\mathcal{A}} \rightarrow \XC  \rightarrow 0 \, .$$

Let  $K = \text{Ker}( V_{\mathcal{A}} \rightarrow
\XC)$. Then,
$$ K \rightarrow \mathcal{X}$$
is a vector bundle of rank $a-1$.
Since the image of
$$\gamma: V_{\mathcal{B}}\rightarrow V_{\mathcal{A}}$$
  lies in $K$,
$\mathcal{M}$ is in fact the zero locus of
$$Q \rightarrow q^*(K)$$
on $V^*_{\mathcal{B}}$.
The dimension of every irreducible
component of $\mathcal{M}$ is at least
$$d+b -a+1 = d + \chi(\mathcal{L}) +1$$
by dimension theory as before.
\end{proof}

\subsection{Irreducible components in the boundary}
\label{xxddx}
We consider here the holomorphic case:
 $\mu=(m_1,\ldots,m_{n})$  is  a vector of zero multiplicities
satisfying
$$m_i \geq 0\, , \ \   \ \sum_{i=1}^n m_i =2g-2\ .$$
Let  $Z\subset \widetilde{\H}_g(\mu)$ be an irreducible component
of dimension $2g-3+n$ supported
in the boundary $\partial\overline{\mathcal{M}}_{g,n}$, and let
$$[C,p_1,\ldots,p_n] \in Z$$
be a generic element with associated twist $I$.
The
dual graph $\Gamma_C$ of $C$
must be of the form described in Section \ref{efa} and Lemma \ref{genn}:
$\Gamma_C$ is  a simple star
with a center vertex $v_0$,
edges
connecting
$v_0$ to
outlying vertices $v_1,\ldots, v_k$,
and a distribution of the parts of $\mu$ (with all negative parts distributed
to the center $v_0$).

By Lemma \ref{sqqr}, there are no irreducible components $Z\subset
\widetilde{\H}_g(\mu)$ supported in the boundary with generic
dual graphs with {\em no} outlying vertices.
All loci with such generic dual graphs are in the closure of $\H_g(\mu)$. A
similar result holds in case there is a single outlying vertex.

\begin{proposition} \label{sqqr2}
If all parts of $\mu$ are non-negative,
there are no irreducible components $$Z\subset
\widetilde{\H}_g(\mu)$$ supported in the boundary with generic
dual graph having exactly $1$ outlying vertex.
All such loci are in the closure of $\H_g(\mu)$.
\end{proposition}

\begin{proof}
As in the proof of Proposition \ref{dim2}, we will study the versal deformation space
of $C$ where
$$[C,p_1,\ldots,p_n] \in Z$$
is a generic element. However, a more delicate argument is required here.
We assume the generic twist $I$ is nontrivial. If the twist $I$ is
trivial, then $[C,p_1,\ldots,p_n]$ is a usual canonical
divisor and the proof of Proposition \ref{dim2} directly applies.
The nodes of $C$ corresponding to self-edges of $v_0$
cannot be twisted by $I$.

We assume{\footnote{In case $g=0$
vertices occur, we can rigidify with $q$ markings as in
the proof of Proposition \ref{dim2}. We leave the routine
modification to the reader.}}  the vertices $v_0$ and $v_1$ both
have genus $g_i>0$.
Let $\mathcal{V}$ be the versal deformation space of
the unpointed nodal curve $C$,
$$\text{dim}\, \mathcal{V} = 3g-3\ . $$
let $e_1$ be the number of edges connecting $v_0$ and $v_1$.
The subvariety $\mathcal{S}\subset \mathcal{V}$
of curves preserving the nodes of $C$ corresponding to the
edges connecting $v_0$ to $v_1$ is
of codimension $e_1$.

We blow-up $\mathcal{V}$ in
the subvariety $\mathcal{S}$,
$$\nu:\mathcal{V}_1 \rightarrow \mathcal{V}\, .$$
The fiber of $\mathcal{V}_1$ above the original point
$[C]\in \mathcal{V}$ is
$$\nu^{-1}([C])=\mathbb{P}^{e_1-1} \, .$$

Via pull-back, we have a universal family of curves
$$\pi: \mathcal{C} \rightarrow \mathcal{V}_1\, . $$
The blow-up yields a nonsingular exceptional divisor
$$E_1\subset \mathcal{V}_1\, $$
which contains
$\nu^{-1}([C])$.
By locally base-changing $\mathcal{V}_1$ via the $r^{th}$-power of the
equation of $E_1$, we can introduce transverse $A_{r-1}$-singularities
along the nodes of the fibers of $\mathcal{C}$ corresponding
to the edges $e_1$. The crepant resolution introduces
a chain of $r-1$ rational curves.

The outcome after such a base change is a family
$$\widetilde{\pi}:\widetilde{C} \rightarrow {\mathcal{V}}_1$$
with the following properties over the generic point
$\xi\in \nu^{-1}([C])$:
\begin{enumerate}
\item[(i)] $\widetilde{C}_\xi$ is a semistable curve obtained from $C$
by inserting chains of $r-1$ rational components at the nodes,
\item[(ii)] the total space of $\widetilde{C}$ is nonsingular near $\widetilde{C}_\xi$,
\item[(iii)] the components of $\widetilde{C}_\xi$ correspond to
{\em divisors} in the total space of $\widetilde{C}$.
\end{enumerate}

After selecting $r$ appropriately, we can twist
the total space $\widetilde{C}$ as specified by $I$.
Such a geometric twisting is possible because of condition (iii)
above.

We follow the construction of Section \ref{smm}.
There, a non-negative twist $\gamma_D$ is associated
to each component $D \subset \widetilde{C}_\xi$. More precisely,
the construction assigns
$$\gamma_{v_0}=0\, ,\ \ \ \ \ \gamma_{v_1}= \prod_{i=1}^{e_1} I(q_i,v_0) >0 \ $$
and strictly positive twists less than $\gamma_{v_1}$
for the additional rational components.
By twisting by the divisor $-\gamma_{v_1}\cdot E_1$ near $\xi \in \mathcal{V}_1$,
 we can change the assignments to
\begin{equation}\label{zzz222}
\widehat{\gamma}_{v_0}= -\prod_{i=1}^{e_1} I(q_i,v_0) <0\, , \ \ \ \ \
\widehat{\gamma}_{v_1}=0 \
\end{equation}
and strictly negative twists greater than $\widehat{\gamma}_{v_0}$
for the additional rational components.

Using the twists \eqref{zzz222}, we obtain a line bundle
\begin{equation}\label{x443}
\mathcal{L} =\omega_{\widetilde{\pi}}\left( \sum_{D} \widehat{\gamma}_D \cdot[D] \right)
\end{equation}
where the sum is over the irreducible components of
  $\widetilde{C}_\xi$ for generic $\xi \in \nu^{-1}([C])$.
The line bundle \eqref{x443}
is defined over an open set
$$\mathcal{U} \subset \mathcal{V}_1$$
which contains the generic element $\xi\in \nu^{-1}([C])$.

On the fibers of $\widetilde{\pi}$ over $\mathcal{U}$,
$$H^1\Big(\omega_{\widetilde{\pi}}\Big( \sum_{D} \widehat{\gamma}_D \cdot[D] \Big)
\Big)
\stackrel{\sim}{=}
H^0\Big(\mathcal{O}\Big( \sum_{D} -\widehat{\gamma}_D \cdot[D] \Big)\Big)^{\vee}$$
by Serre duality.
Since $-\widehat{\gamma}_D\geq 0$, we obtain
a canonical section
$$\mathbb{C} \rightarrow H^0\Big(\mathcal{O}\Big( \sum_{D} -\widehat{\gamma}_D \cdot[D] \Big)\Big)^{\vee}\, .$$
Since $-\widehat{\gamma}_{v_1} = 0$, the
canonical section does not vanish identically
 on {\em any} fiber of $\widetilde{\pi}$
over $\mathcal{U}$.
Hence, we obtain a canonical quotient
$$H^1\Big(\omega_{\widetilde{\pi}}\Big( \sum_{D} \widehat{\gamma}_D \cdot[D] \Big)\Big)
\rightarrow \mathbb{C} \rightarrow 0\ .$$

We now can estimate the dimension of the space of good sections
of $\mathcal{L}$ on the fibers of $\widetilde{\pi}$ over $\mathcal{U}$
near the generic element of $\nu^{-1}([C])$
 using Proposition \ref{qq2} of Section \ref{qq}.
The dimension of the space of sections
is at least
$$\text{dim}\, \mathcal{U} + \chi(C,L)+1\, =\,  3g-3+g-1+1 \, =
\,  4g-3\, .$$
Hence, dimension of the space $\mathcal{T}$ of twisted
canonical divisors on the fibers of $\widetilde{\pi}$
over $\mathcal{U}$ corresponding to twisted canonical line bundle $\mathcal{L}$
near the generic element of  $\nu^{-1}([C])$ is at least $4g-4$.

We must further impose conditions on the twisted
canonical divisors parameterized by $\mathcal{T}$
to obtain the shape $\sum_{i=1}^n m_i p_i$.
These conditions impose at most
$$2g-2-n$$
constraints. Hence, we conclude the dimension of
space $\mathcal{T}(\mu)$ of twisted
canonical divisors on the fibers of
\begin{equation*}
\widetilde{\pi}: \widetilde{\mathcal{C}} \rightarrow \mathcal{U}\,
\end{equation*}
near $[C, p_1, \ldots, p_n]$ of shape $\mu$ is {\em at least}
 $$4g-4 - (2g-2 -n)= 2g-2+n\, .$$
The result contradicts the dimension $2g-3+n$ of $Z$.
\end{proof}

\vspace{10pt}

By Lemma \ref{sqqr} and Proposition \ref{sqqr2}
in the holomorphic case, the loci
$$Z\subset \widetilde{\H}_g(\mu)$$
corresponding to star graphs with 0 or 1 outlying vertices
are contained in the closure of $\H_g(\mu)$.
However, there are virtual components of $\widetilde{\H}_g(\mu)$
with more outlying vertices.

A simple example occurs in genus $g=2$ with $\mu=(1,1)$.
A nonsingular pointed curve $$[C,p_1,p_2] \in \mathcal{M}_{2,2}$$
lies in $\H_2(1,1)$ if and only if
$\{p_1,p_2\}\subset C$ is a fiber of
the unique degree 2 map
$$ C \rightarrow \mathbb{P}^1$$
determined by the canonical series.
An easy argument using admissible covers shows that
the generic element of the virtual component corresponding
to the star graph below is not in the closure of $\H_2(1,1)$.

\begin{centering}{\hspace{50pt} \includegraphics[scale=0.5]{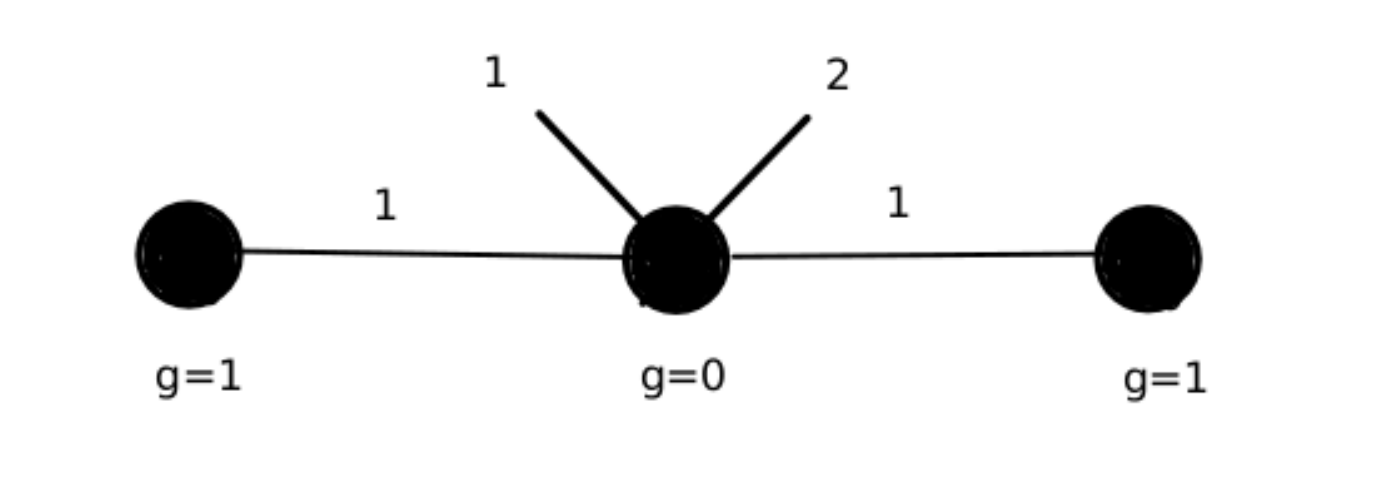}}
\end{centering}

\noindent Both parts of $\mu$ are distributed to the center.
The genera of the vertices and the twists of the edges are
specified in the diagram.
The above example (which occurs even in compact type)
was discussed earlier by Chen \cite{Ch}.

\section{Twisted canonical divisors and limits of theta characteristics} \label{tthh}

\subsection{Theta characteristics}
We illustrate the results of the paper (particularly Theorem \ref{main1}
 and Proposition \ref{sqqr2})
in two much studied cases linked to the classical theory of
theta characteristics.  Assume $\mu=(m_1, \ldots, m_n)$
consists of {\em even} non-negative parts,
$$m_i=2a_i\, , \ \ \ \ \ \ \ \ \ \sum_{i=1}^n a_i =g-1\, .$$
 We write $\mu=2\alpha$ where $\alpha\vdash g-1$ .

If $[C, p_1, \ldots, p_n]\in \H_g(2\alpha)$ is a canonical divisor, then
$$\eta\w= \OO_C(a_1 p_1+\cdots +a_n p_n)$$
is a theta characteristic.
The study of the closure $\overline{\H}_g(2\alpha)$  is therefore
intimately related to the geometry of the moduli space $\ss_g$ of spin curves of genus $g$, see \cite{C}.

Two cases are of particular interest. If $\mu=(2g-2)$ is of length one,
$\H_g(2g-2)$ is the locus of \emph{subcanonical} pointed curves:
 $$[C,p]\in \cM_{g,1} \ \ \ \Leftrightarrow \ \ \
\omega_C \ \stackrel{\sim}{=}\  \mathcal{O}_C((2g-2)p)\, .$$
Subcanonical points are extremal Weiertrass points on the curve.
 For $g\geq 4$,
Kontsevich and Zorich \cite{KZ} have shown that the space $\H_g(2g-2)$ has three connected components (of dimension $2g-1$):
\begin{itemize}
\item the locus $\H_g(2g-2)^{+}$ of curves $[C,p]$ for which
 $\OO_C((g-1)p)$ is an even theta characteristic,
\item the locus $\H_g(2g-2)^{-}$ of curves $[C,p]$ for which $\OO_C((g-1)p)$ is an odd theta characteristic,
\item the locus $\H_g(2g-2)^{\mathrm{hyp}}$ of curves $[C,p]$ where $C$ is hyperelliptic and $p\in C$ is a Weierstrass point.
\end{itemize}
For $g=3$, the first and the third coincide
$$\H_3(4)^{+}=\H_3(4)^{\mathrm{hyp}}\, ,$$
 thus $\H_3(4)$ has only two connected components.
For $g=2$, the space $\H_2(2)$ is irreducible. The geometry of the
compactifications of these loci in small genera has been  studied in \cite{ChT}.

The second case we consider is $\mu=(2, \ldots, 2)$ of length $g-1$.
The space $\H_g(\underline{2})=\H_g(2, \ldots, 2)$ splits into two connected components:

\begin{itemize}
\item the locus $\H_g(\underline{2})^-$ of curves
$[C, p_1, \ldots, p_{g-1}]\in \cM_{g,g-1}$  for which
 $$\eta\w=\OO_C(p_1+\cdots+p_{g-1})$$ is an odd theta characteristic,
\item the locus $\H_g(\underline{2})^+$  of curves
$[C, p_1, \ldots, p_{g-1}]\in \cM_{g,g-1}$ for which
$$\eta\w=\OO_{C}(p_1+\cdots +p_{g-1})$$ is a vanishing theta-null (even and
 $h^0(C, \eta)\geq 2$).
\end{itemize}
The component $\H_g(\underline{2})^-$ is a generically finite cover of $\cM_g$.
For example, $\H_3(\underline{2})^-$ is birationally isomorphic to the canonical double cover over the moduli space of bitangents of nonsingular quartic curves. On the other hand,
 the
component $\H_g(\underline{2})^+$ maps with 1 dimensional fibres over the divisor $\Theta_{\mathrm{null}}$ of curves with a vanishing theta-null, see \cite{F}.

\subsection{Spin curves}
We require a few basic definitions concerning spin curves.
A connected nodal curve $X$ is \emph{quasi-stable} if, for every
 component
$$E\stackrel{\sim}{=} {\mathbb{P}}^1\subset X\, , $$
 the following two properties are satisfied:
\begin{enumerate}
\item[$\bullet$]
 $k_E=|E\cap \overline{(X-E)}|\geq 2$,
\item[$\bullet$] rational components $E, E'\subset X$ with $k_E=k_{E'}= 2$ are always disjoint.
\end{enumerate}
The irreducible components $E\subset X$ with $k_E=2$ are \emph{exceptional}.

\begin{definition}[Cornalba \cite{C}] \label{stspin}
 A stable \emph{spin curve} of genus
$g$ consists of a  triple $(X, \eta, \beta)$ where
\begin{enumerate}
\item[(i)]
 $X$ is a genus
$g$ quasi-stable curve,
\item[(ii)]
$\eta\in \mathrm{Pic}^{g-1}(X)$ is a line
bundle of total degree $g-1$ with $\eta_{E}=\OO_E(1)$ for all
exceptional components $E\subset X$,
\item[(iii)] $\beta:\eta^{\otimes
2}\rightarrow \omega_X$ is a homomorphism of sheaves which is generically
nonzero along each nonexceptional component of $X$.
\end{enumerate}
\end{definition}

If  $(X, \eta, \beta)$ is a spin curve with exceptional components
$E_1, \ldots, E_r\subset X$, then $\beta_{E_i}=0$ for $i=1, \ldots, r$
by degree considerations. Furthermore,  if
$$\widetilde{X}=\overline{X-\bigcup_{i=1}^r E_i}$$ is viewed as a subcurve
of $X$, then we have an isomorphism
$$
\eta^{\otimes 2}_{\widetilde{X}}\stackrel{\sim}\longrightarrow
\omega_{\widetilde{X}}.
$$

The moduli space $\ss_g$ of stable spin curves of genus $g$ has been constructed by Cornalba \cite{C}.
There is a proper morphism $$\pi:\ss_g\rightarrow \mm_g$$ associating to a spin curve $[X, \eta, \beta]$ the curve obtained from $X$ by contracting all exceptional components.
The parity $h^0(X, \eta) \mbox{ mod 2}$ of a spin curve is invariant under deformations.
The moduli space
 $\ss_g$  splits into two connected components $\ss_g^+$ and $\ss_g^-$ of relative degree $2^{g-1}(2^g+1)$ and
$2^{g-1}(2^g-1)$ over $\overline{\mathcal{M}}_g$ respectively.  The birational geometry of $\ss_g$ has been studied in \cite{F}.

\subsection{Twisted canonical divisors on 2-component curves.}
We will describe the limits of supports of theta characteristics  in
case the underlying curve
$$C=C_1 \cup C_2$$ is a union of two nonsingular curves $C_1$ and $C_2$ of
genera $i$ and $g-i-\ell+1$  meeting transversally in a set of $\ell$ distinct points
$$\Delta=\{x_1, \ldots, x_{\ell}\} \subset C\, .$$
Let  $2a_1 p_1+\cdots +2a_n p_n$ be a twisted canonical divisor on $C$,
$$[C,p_1,\ldots,p_n] \in \widetilde{\H}_g(2\alpha)\, .$$
By definition, there exist twists
$$I(x_j, C_1)=-I(x_j, C_2)\, , \ \ \ \  j=1, \ldots, \ell\, $$
 for which the following linear equivalences on $C_1$ and $C_2$ hold:
$$\omega_{C_1}\equiv \sum_{p_i\in C_1} 2a_i p_i-\sum_{j=1}^{\ell} \Bigl(I(x_j,C_1)+1\Bigr) x_j\  \mbox{ and } \ \omega_{C_2}\equiv \sum_{p_i\in C_2} 2a_i p_i-\sum_{j=1}^{\ell} \Bigl(I(x_j,C_2)+1\Bigr) x_j.$$
By Proposition \ref{sqqr2}, all twisted holomorphic differentials on $C$
are smoothable,
  $$[C,p_1, \ldots, p_n]\in \hh_g(\mu)\, .$$

Following \cite{C}, we describe all spin structures having $C$ as underlying stable curve.
If $[X, \eta, \beta] \in \pi^{-1}([C])$, then the number of nodes of $C$ where no exceptional component is inserted must be \emph{even}. Hence,
$X$ is obtained from $C$ by {\em blowing-up} $\ell-2h$ nodes.{\footnote{The
term {\em blowing-up} here just means inserting an exceptional $E\stackrel{\sim}{=}\mathbb{P}^1$.}}
Let  $\{x_1, \ldots, x_{2h}\}\subset C$
be the non blown-up nodes. We have
 $$X=C_1\cup C_2 \cup E_{2h+1}\cup \ldots \cup E_{\ell}\, ,$$
where $C_1\cap E_i=\{x'_i\}$ and $C_2\cap E_i=\{x''_i\}$ for $i=2h+1, \ldots, \ell$,
see Figure 2.
Furthermore,
$$\eta_{E_i}\stackrel{\sim}{=}\OO_{E_i}(1)\, , \ \ \eta_{C_1}\in \mbox{Pic}^{i-1+h}(C_1)\, , \ \
\eta_{C_2}\in \mbox{Pic}^{g-i-\ell+h}(C_2)$$ with the latter  satisfying
$$\eta_{C_1}^{\otimes 2}\w=\omega_{C_1}(x_1+\cdots+x_{2h})\ \ \mbox{ and }\ \ \eta_{C_2}^{\otimes 2}\w=\omega_{C_2}(x_1+\cdots+x_{2h}).$$

\begin{figure}[h]
 \centering
  \includegraphics[width=4cm, height=5cm]{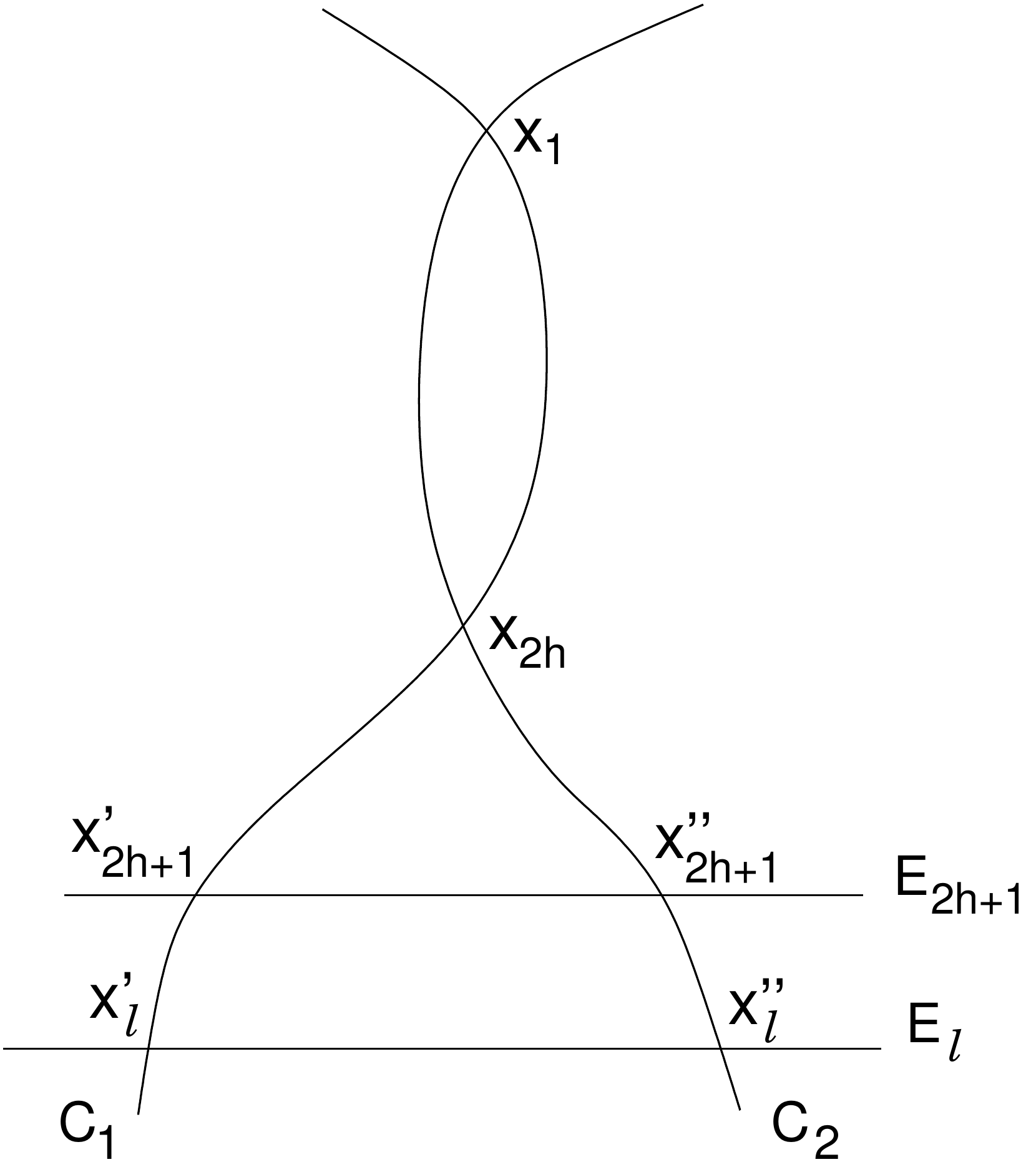}
  \caption{The quasi-stable curve $X$}
\end{figure}

There is a precise correspondence between
the spin structures of Cornalba \cite{C} and
 the system of twists associated to
a point in $\widetilde{\H}_g(2\alpha)$.

\begin{proposition}\label{bananas}
Let $C=C_1 \cup C_2$ as above.
There exists a spin structure in $\ss_g$ corresponding to $[C,p_1,\ldots,p_n] \in
\widetilde{\H}_g(2\alpha)$ with twist $I$, which 
is obtained by blowing-up $C$ at all the nodes $x_j$ where the twist
$$I(x_j, C_1)=-I(x_j,C_2)$$ is odd.
\end{proposition}
\begin{proof}
Since $[C, p_1, \ldots, p_n]\in \hh_g(2\alpha)$ by Proposition \ref{sqqr2},
there exists an $n$-pointed family of stable spin curves  over a  disc
$$\mathcal{X}\rightarrow \Delta\, , \ \ \ \eta\in \mbox{Pic}(\mathcal{X})\, , \ \ \ p_1, \ldots, p_n:\Delta\rightarrow \mathcal{X},$$ satisfying
\begin{enumerate}
\item[$\bullet$]
 for $t\in \Delta^\star$, the fibre $X_t$ is nonsingular
and $$\eta_{X_t}\stackrel{\sim}{=}\OO_{X_t}\bigr(a_1 p_1(t)+\cdots+a_n p_n(t)\bigl)\, ,$$
\item[$\bullet$]  the central fibre $X_0$ is obtained from $C$ by inserting  nonsingular
 rational components $E_{2h+1}, \ldots, E_{\ell}$ at some of the nodes of $C$
which we assume to be $x_{2h+1}, \ldots, x_{\ell}$.
\item[$\bullet$]
$\eta_{E_j}\stackrel{\sim}{=}\OO_{E_j}(1)$  for $j=2h+1, \ldots, \ell$ .
\end{enumerate}

 By carrying out an argument following Section \ref{gentw},
after a finite base change and after resolution of singularities,
we arrive at a new family
$$\widetilde{\pi}:\widetilde{\mathcal{X}}\rightarrow \Delta\, ,$$
 with central fibre
$$\mu:\widetilde{\pi}^{-1}(0)\rightarrow X_0$$
obtained from $X_0$ by inserting chains of nonsingular rational curves at the nodes of $X_0$
for which
the line bundle $$\mu^*(\eta_X)(-a_1 p_1-\cdots -a_n p_n)$$
 is the restriction to the central fibre of a line bundle on $\widetilde{\mathcal{X}}$ supported {\em only}
on the irreducible components of $\widetilde{\pi}^{-1}(0)$.

Just as in Section \ref{gentw},
we then obtain integral twists
$$\tau_j=J(x_j,C_1)=-J(x_j, C_2)$$
 for $j=1, \ldots, 2h$ and
$$\tau_{j}'=J(x_j', C_1)=-J(x_j', E_j) \ \ \ \text{and}\ \ \
\tau_{j}''=J(x_j'',E_j)=-J(x_j'',C_2)$$
 for $j=2h+1,\ldots, \ell$, for which
 the restrictions of $\eta_X\bigl(-\sum_{i=1}^n a_i p_i\bigr)$ to the irreducible
components of $X_0$
are isomorphic to the bundles given by the twists $J$ at the incident nodes. In particular,
$$\eta_{E_j}\w=\OO_{E_j}(1)=\OO_{E_j}(-\tau_j'+\tau_j''),$$
hence $\tau_j''=\tau_j'+1$ for $j=2h+1,\ldots, \ell$. Furthermore, we have
$$\eta_{C_1}\Bigl(-\sum_{p_i\in C_1} a_i p_i\Bigr)\w=\OO_{C_1}\Bigl(\tau_1 x_1+\cdots+\tau_{2h} x_{2h}+\tau_{2h+1}' x'_{2h+1}+\cdots +\tau_{\ell}' x_{\ell}'\Bigr)\, ,$$
$$\eta_{C_2}\Bigl(-\sum_{p_i\in C_2} a_i p_i\Bigr)\w=\OO_{C_2}\Bigl(-\tau_1 x_1-\cdots-\tau_{2h} x_{2h}-\tau_{2h+1}'' x_{2h+1}''-\cdots -\tau_{\ell}'' x_{\ell}''\Bigr).$$
By squaring these relations and comparing to the original $I$ twists,
we conclude
$$I(x_j, C_1)=-2\tau_j$$
for $j=1, \ldots, 2h$ and
$$I(x_j,C_1)=-2\tau_j'-1$$ for $j=2h+1, \ldots, \ell$ respectively.
\end{proof}

\subsection{Subcanonical points}
\subsubsection{Genus $3$.} We illustrate Proposition \ref{bananas} in
case{\footnote{The case has also been treated by D. Chen in
\cite[Section 7.2]{Ch}.}
$g=3$ where  $$\H_3(4)=\H_3(4)^{-} \sqcup  \H_3(4)^{\mathrm{hyp}}\, .$$  Assume
$C_1$ and $C_2$ are both nonsingular elliptic curves and $p\in C_1\setminus\{x_1, x_2\}$.
Suppose
$$[C_1\cup C_2, p]\in \hh_{3}(4)\, .$$
The only possible twists are
$$I(x_1, C_1)=I(x_2, C_1)=1\, .$$
The curve $[C_1\cup C_2, p]\in \mm_{3,1}$ lies in $\hh_3(4)$ if and only if the
linear equivalence
\begin{equation}\label{gen3}
4p\equiv 2x_1+2x_2
\end{equation}
holds on $C_1$.

The associated spin structure $[X, \eta]$ is obtained by inserting at both  $x_1$ and $x_2$ rational components
 $E_1$ and $E_2$,
$$X=C_1\cup C_2\cup E_1\cup E_2$$
with the intersections
 $$C_1\cap E_1=\{x_1'\}\, , \ \ C_1\cap E_2=\{x_2'\}\, , \ \  C_2\cap E_1=\{x_1''\}\, ,\ \
 C_2\cap E_2=\{x_2''\}\, .$$
 The line bundle $\eta$ satisfies
$$\eta_{E_i}\w=\OO_{E_i}(1)\, , \ \  \eta_{C_1}\w=\OO_{C_1}(2p-x_1-x_2)\, , \ \
 \eta_{C_2}\w=\OO_{C_2}\, .$$
An argument via admissible covers implies
 $$[C,p]\in \hh_3(4)^{\mathrm{hyp}} \ \ \ \Leftrightarrow \ \  \ \OO_{C_1}(2p-x_1-x_2)\w=\OO_C\, .$$
In the remaining case, where
$$\eta_{C_1}\w=\OO_{C_1}(2p-x_1-x_2)\in JC_1[2]-\{\OO_{C_1}\}$$ is a
nontrivial $2$-torsion point, $\eta_{C_1}$ is an even and $\eta_{C_2}$ is an odd spin structure, so $[X, \eta]\in \ss_3^{-}$, and accordingly
$[C,p]\in \hh_3(4)^{-}$.




\subsubsection{Genus $4$.} Consider next genus $4$ where
$$\H_4(6)=\H_4(6)^+ \sqcup \H_4(6)^- \sqcup \H_4(6)^{\mathrm{hyp}}$$ has three connected components
of dimension $7$. Let
$$C=C_1\cup C_2\, ,  \ \ \ \ g(C_1)=1\, ,\  \ g(C_2)=2$$
with $C_1\cap C_2= \{ x_1,x_2\} $.
Suppose $p\in C_1$  and
$$[C_1\cup C_2, p]\in \hh_{4}(6)\, .$$
There are two possible twists.

\vspace{12pt}
\noindent $\bullet$\, Twist $I(x_1, C_1)=I(x_2, C_1)=2$.
\vspace{12pt}

We then obtain the following linear equivalences on $C_1$ and $C_2$:
\begin{equation}\label{gen4}
\OO_{C_1}(6p-3x_1-3x_2)\w=\OO_{C_1} \ \mbox{ and } \ \ \omega_{C_2}\w=\OO_{C_2}(x_1+x_2).
\end{equation}
Each equation appearing in (\ref{gen4}) imposes a codimension 1 condition
on the respective Jacobian.
The corresponding spin curve in $[X, \eta]\in \ss_4$ has
 underlying model
$X=C$ (no nodes are blown-up), and the
spin line bundle $\eta\in \mbox{Pic}(X)$ satisfies
$$\eta_1\w=\OO_{C_1}(3p-x_1-x_2) \ \ \ \text{and}\ \ \  \eta_2\w=\OO_{C_2}(x_1+x_2)\, .$$
 As expected, $\eta_i^{\otimes 2}\w=\omega_{C_i}(x_1+x_2)$ for $i=1,2$.

We can further distinguish between the components of the closure $\hh_4(6)$. We have
 $[C,p]\in \hh_4(6)^{\mathrm{hyp}}$ if and only if
$$
\OO_{C_1}(2p-x_1-x_2)\w=\OO_{C_1} \mbox{ and } \  \omega_{C_2}\w=\OO_{C_2}(x_1+x_2)\, ,
$$
whereas $[C,p]\in \hh_4(6)^-$ if and only if
$$
\OO_{C_1}(2p-x_1-x_2)\in JC_1[3]-\{\OO_{C_1}\} \ \mbox{ and } \ \omega_{C_2}\w=\OO_{C_2}(x_1+x_2)\, .
$$

\vspace{12pt}
\noindent $\bullet$\, Twist $I(x_1, C_1)=3$ and $I(x_2, C_1)=1$.
\vspace{12pt}

We then obtain the following
linear equivalences on $C_1$ and $C_2$:
\begin{equation}\label{gen42}
\OO_{C_1}(6p-4x_1-2x_2)\w=\OO_{C_1} \ \mbox{ and } \ \omega_{C_2}\w=\OO_{C_2}(2x_1)\, .
\end{equation}
As before, the
conditions in (\ref{gen42}) are both codimension $1$ in the respective Jacobians.
The corresponding spin curve $[X, \eta]$ is obtained by blowing-up both nodes
$x_1$ and $x_2$,
$$X=C_1\cup C_2 \cup E_1 \cup E_2,$$
with the intersections
$$C_1\cap E_1=\{x_1'\}\, , \ \ C_1\cap E_2=\{x_2'\}\, , \ \
C_2\cap E_1=\{x_1''\}\, , \ \ C_2\cap E_2=\{x_2''\}\, .$$
The line bundle $\eta$ satisfies
 $$\eta_{E_i}\w=\OO_{E_i}(1)\, , \ \
\eta_{C_1}\w=\OO_{C_1}(3p-2x_1-x_2)\, , \ \  \eta_{C_2}\w=\OO_{C_2}(x_1)\, .$$
We have  $[C,p]\in \hh_4(6)^+$ if and only if
$\eta_{C_1}\w=\OO_{C_1}$. If $\eta_{C_1}$ is a nontrivial $2$-torsion point, then $[C,p]\in
\hh_4(6)^-$.

If $p\in C_2$ and $[C_1\cup C_2, p]\in \hh_{4}(6)$,
then  $I(x_1, C_1)=I(x_2, C_1)=-1$ is the only possible twist, and
\begin{equation}
\omega_{C_2}\w=\OO_{C_2}(6p-2x_1-2x_2)\, .
\end{equation}
The corresponding spin structure is realized on the curve $X$ obtained by inserting rational components at both $x_1,x_2\in C$ and
$$\eta_{C_1}\w=\OO_{C_1}\, , \ \  \ \eta_{C_2}\w=\OO_{C_2}(3p-x_1-x_2)\, .$$
 We have  $[C,p]\in \hh_4(6)^{+}$ (respectively $\hh_4(6)^{-}$) if and only if $\eta_{C_2}$ is an odd (respectively even) theta characteristic on $C_2$.

\subsection{Limits of bitangents  in  genus $3$}
\subsubsection{Two elliptic components}
Recall the decomposition into components in the case of genus $3$:
$$\H_3(2,2)=\H_3(2,2)^-\sqcup \H_3(2,2)^+,$$
where the general point $[C, p_1, p_2]$ of $\H_3(2,2)^-$ corresponds
to a bitangent line $$\langle p_1, p_2\rangle\in (\mathbb P^2)^{\vee}$$
of the  canonical model $C\subset \mathbb P^2$ and
 the general point of $\H_3(2,2)^+$ corresponds to a hyperelliptic curve $C$
 together with two points lying a fibre of the hyperelliptic pencil.

Let $C=C_1\cup C_2$ where $C_1$ and $C_2$ are nonsingular
 elliptic curves with
$$C_1\cap C_2=\{x_1, x_2\}\, .$$
 Suppose $[C, p_1, p_2]\in \hh_3(2,2)$.  There are two cases.

\vspace{8pt}
\noindent {\bf Case 1}:\,  $p_1$ and $p_2$ lie on the same component of
$C$, say
$p_1, p_2\in C_1\setminus\{x_1,x_2\}$.

\vspace{8pt}
Then
$I(x_1, C_1)=I(x_2, C_1)=1$, and  the following equation holds:
\begin{equation}\label{bitg1}
\OO_{C_1}(2x_1+2x_2-2p_1-2p_2)\w=\OO_{C_1}\, .
\end{equation}
The induced spin structure is obtained by blowing-up both nodes $x_1$ and $x_2$, and the corresponding spin bundle $\eta$ restricts as
$$\eta_{C_1}\w=\OO_{C_1}(p_1+p_2-x_1-x_2) \ \ \mbox{ and }\ \
\eta_{C_2}\w=\OO_{C_2}.$$
If the linear equivalence
$p_1+p_2\equiv x_1+x_2$ holds on $C_1$,
then $[C, p_1, p_2]\in \hh_3(2,2)^+$ as can be seen directly via admissible covers.\footnote{One can construct an admissible cover $f:C\rightarrow R$, where $R=(\mathbb P^1)_1\cup _t (\mathbb P^1)_2$, with restrictions $f_{| C_i}:C_i\rightarrow (\mathbb P^1)_i$ being the morphisms induced by the linear system $|\mathcal{O}_{C_i}(x_1+x_2)|$ for $i=1,2$ and $f^{-1}(t)=\{x_1,x_2\}$. Then clearly $f(p_1)=f(p_2)$.} If
$$\eta_{C_1}\in JC_1[2]-\{\OO_{C_2}\}\, ,$$
then, since $\eta_{C_2}$ is odd, we have $[C, p_1, p_2]\in \hh_3(2,2)^-$.
There is a $1$-dimensional family of limiting bitangents to $C$, which is to be expected since $C$ is hyperelliptic.

\vspace{8pt}
\noindent {\bf Case 2}:\,
$p_i\in C_i\setminus \{x_1, x_2\}$ for $i=1,2$.
\vspace{8pt}

 Then $I(x_1, C_1)=I(x_2, C_1)=0$. No node is twisted, so
no rational components are inserted for the induced spin structure.
The spin bundle $\eta\in \mbox{Pic}^2(C)$ has restrictions
$$\eta_{C_i}\w=\OO_{C_i}(p_i)\, $$
 for $i=1,2$.
Since the resulting spin curve is odd,
$$[C, p_1, p_2]\in \hh_3(2,2)^-\, .$$
We obtain $16=4\times 4$ limits of bitangents
corresponding to the choice of points
$p_i\in C_i$ satisfying
$$\OO_{C_i}(2p_i-x_1-x_2)\w=\OO_{C_i}\, .$$

\vspace{12pt}

We now study three further degenerate cases to offer an illustration of the transitivity condition in the definition of twists.

\vspace{8pt}
\noindent {\bf Case 3}:\,
$[C_1\cup C_2\cup E, p_1, p_2]\in \widetilde{\H}_3(2,2)$
with $p_1\in E\w= {\mathbb{P}}^1$ ,  $p_2\in C_1$.

\vspace{8pt}
We have the following component intersections:
$$C_1\cap C_2=\{x_2\}\, , \ \ C_1\cap E=\{x_1'\}\,  , \ \
C_2\cap E=\{x_1''\}\, .$$
The twists are subject to three dependent equations.
Let $a=I(x_1',C_1)$. We find
$$I(x_1'',C_2)=-I(x_1'',E)=-2-a\  \ \ \mbox{ and }\ \ \ I(x_2,C_1)=-I(x_2,C_2)
=-a\, .$$
If $a\leq -2$,
the ordering on components defined by $I$,
$$C_1<E\leq C_2<C_1\,\ ,$$
contradicts transitivity. Similarly, if $a\geq 0$, we obtain
$$E\leq C_1\leq C_2<E\, ,$$ which is also a contradiction.
Therefore $a=-1$, and we have
$$\OO_{C_1}(2p_2-2x_2)\w=\OO_{C_1}.$$
The result is a $4$-dimensional subvariety of $\mm_{3,2}$ and therefore
 cannot be a component of $\widetilde{\H}_3(2,2)$.

\vspace{8pt}
\noindent {\bf Case 4}:\,
$\bigl[C_1\cup C_2\cup E_1\cup E_2, p_1, p_2]\in \widetilde{\H}_3(2,2)$
with $p_i\in E_i\w={\mathbb{P}}^1$.
\vspace{8pt}

We have the following component intersections:
$$C_1\cap E_1=\{x_1'\}\, , \ \ C_1\cap E_2=\{x_2'\}\, ,\ \
 C_2\cap E_1=\{x_1''\}\, , \ \ C_2\cap E_2=\{x_2''\}\, .$$
Let $a=I(x_1',C_1)=-I(x_1',E_1)$. We obtain,
after solving a system of equations,
$$I(x_2',C_1)=-I(x_2',E_2)=-2-a\, ,\ \ \  I(x_2'',E_2)=-I(x_2'',C_2)=-a\, ,$$
$$I(x_1'',E_1)=-I(x_1'',C_2)=a+2\, .$$
When $a\leq -2$, we obtain  $$C_1<E_1\leq C_2\leq E_2\leq C_1\, ,$$
a contradiction. The case $a\geq 0$ is similarly ruled out, which forces $a=-1$. Then the twists impose no further constraints on the point in $\mm_{3,2}$, so
 we obtain again a $4$-dimensional subvariety of $\widetilde{\H}_3(2,2)$ which cannot be an irreducible component.

\vspace{8pt}
\noindent {\bf Case 5}:\,
$[C_1\cup C_2\cup E, p_1, p_2]\in \widetilde{\H}_3(2,2)$ with
$p_1\neq p_2\in E\w={\mathbb{P}}^1$.
\vspace{8pt}

We have the following component intersections:
$$C_1\cap C_2=\{x_2\}\, , \ \ C_1\cap E=\{x_1'\}\, , \ \ C_2\cap E=\{x_1''\}\, .$$
Arguments as above yield the conditions
$$\OO_{C_1}(2x_1'-2x_2)\w=\OO_{C_1}\ \  \mbox{ and } \ \ \OO_{C_2}(2x_1''-2x_2)\w=\OO_{C_2}\, $$
which constrains the locus to 3 dimensions in $\mm_{3,2}$.

\subsubsection{Singular quartics}

A nonsingular plane quartic curve
$$C\subset \mathbb P^2$$
is of genus $3$ and
has $28$ bitangent lines corresponding to the set $\theta(C)$
of $28$
odd theta characteristics.
A general quartic
$C\subset \mathbb P^2$ can be reconstructed from its $28$ bitangent
lines \cite{CS}:
 the rational map
$$\mm_3\ \ \dashrightarrow \ \ \mathrm{Sym}^{28} (\mathbb P^2)^{\vee}
\dblq\,
\mathbb{PGL}(3)\, , \ \ \ \ \ C\mapsto \theta(C)$$
is generically injective.
Our approach to
  $\hh_3(2,2)$ naturally recovers classical results on
the limits of bitangents for singular quartic curves.
The investigation also further
 illustrates the difference between $\hh_3(2,2)$ and
$\widetilde{\H}_3(2,2)$.

\begin{definition} Let $C\subset \mathbb P^2$ be an irreducible plane quartic. A line $$L\in (\mathbb P^2)^{\vee}$$ is a \emph{bitangent} of $C$ if
the cycle $L\cdot C$ is everywhere nonreduced and even. A bitangent $L$ is
of {\em type} $i$ if $L$ contains precisely $i$ singularities of $C$.
For $i=0,1,2$, let $b_i$ denote the number of bitangent lines to $C$ of type $i$.
\end{definition}

Let $C\subset \mathbb P^2$ be an irreducible plane quartic having $\delta$ nodes and $\kappa$ cusps. By Pl\"ucker's formulas involving the singularities of the dual plane curve, $C$ has
$$b=b_0+b_1+b_2=28+2\delta(\delta-7)+6\kappa \delta+\frac{9\kappa(\kappa-5)}{2}$$
bitangent lines of which $b_2={\delta+\kappa \choose 2}$ are of type $2$,
see \cite{CS} and \cite[Chapter XVII]{Hil}. In Table 1,
the possibilities for  $b_i$, depending on the singularities of $C$,
are listed.

\begin{table}
\begin{center}
\begin{tabular}{|c|c|c|c|}
\hline
 $(\delta, \kappa)$ &  $b_0$ & $b_1$ & $b_2$ \\
\hline
 $(0,0)$ & 28 & 0 & 0\\
 $(1,0)$ & 16 & 6 & 0 \\
 $(2,0)$ & 8 & 8 & 1\\
 $(3,0)$ & 4 & 6 & 3\\
 $(0,1)$ &  10 & 6 & 0\\
 $(0,2)$ & 1 & 6 & 1\\
 $(0,3)$ & 1 & 0 & 3\\
 $(1,1)$ & 4 & 7 & 1\\
\hline
\end{tabular}

\vspace{8pt}
\end{center}
    \caption{Bitangent data for singular quartic curves \cite{CS,Hil}}
    \label{tab:bitangents}
  \end{table}

In order to recover the sum 28 in the  each line of Table 1,
multiplicities must be calculated.
By
 \cite[Lemma 3.3.1]{CS}:
\begin{enumerate}
\item[$\blacktriangleright$]\,
 a bitangent of type $0$ appears with multiplicity $1$ in the scheme
$\theta(C)$,
\item[$\blacktriangleright$]\,
a bitangent of type $1$  containing a node (respectively a cusp) appears in $\theta(C)$ with multiplicity $2$ (respectively $3$),
\item[$\blacktriangleright$]\,
 a bitangent joining two nodes (respectively two cusps) contributes to $\theta(C)$  with multiplicity $4$ (respectively $9$).
\end{enumerate}

\vspace{6pt}
We will explain how the results of Table \ref{tab:bitangents} can be recovered using the geometry of $\widetilde{\H}_3(2,2)$ in two cases (the
others are similiar).

\vspace{8pt}
\noindent {\bf Case 1}:\, $1$-nodal quartics.
\vspace{8pt}

Let $[C, x_1, x_2]\in \cM_{2,2}$ be a general point and denote by
$$C'=C\,/\,x_1\sim x_2$$
the genus $3$ curve obtained by identifying $x_1$ and $x_2$.
If $[C', p_1, p_2]\in \widetilde{\H}_3(2,2)$,
then $$\OO_{C'}(2p_1+2p_2)\stackrel{\sim}{=}\omega_{C'}\,$$
 and  $\OO_C(2p_1+2p_2)\stackrel{\sim}{=}\omega_C(x_1+x_2)$.
Since multiplication by $2$ on $JC$ has degree $4$,
there are $16$ bitangents of type $0$. All 16 lie
in $\overline{\H}_3(2,2)$ by Lemma \ref{sqqr}.

We determine next the bitangents of type $1$. Let
$$[C\cup E, p_1, p_2]\in \widetilde{\H}_3(2,2),$$
where $E\stackrel{\sim}{=}\mathbb{P}^1$, $p_1\in E$,
$p_2\in C$, and $E\cap C=\{x_1, x_2\}$. Then
$$I(x_1, E)=I(x_2,E)=1$$
and $\omega_{C}=\OO_C(2p_2)$. Hence, $p_2$ is one of the $6$
Weierstrass points of $C$.
Thus the bitangent lines of type $1$ join a node and a Weierstrass point of the normalization. Each such bitangent line is counted with multiplicity $2$
since  $[C\cup E, p_1, p_2]$ has an automorphism of order $2$.
By Proposition \ref{sqqr2}, all 6 lie in $\overline{\H}_3(2,2)$.

Finally, we consider the case
$[C\cup E, p_1, p_2]\in \widetilde{\H}_3(2,2)$ where
$$p_1, p_2\in E-\{x_1, x_2\}\, .$$ Then we find
$I(x_1,C)+I(x_2, C)=4\, .$ Both possibilities $\{I(x_1,C), I(x_2, C)\}=\{2,2\}$ and $\{I(x_1, C), I(x_2, C)\}=\{1,3\}$  impose a nontrivial condition on the curve
$[C, x_1, x_2]\in \cM_{2,2}$, so the case does not appear generically.

We have verified the line
corresponding to
$(1,0)$ in Table 1: $b_0=16$, $b_1=6$, $b_2=0$, and
$28= 16 + 2\cdot 6\, .$

\vspace{8pt}
\noindent {\bf Case 2}:\, $1$-cuspidal quartics.
\vspace{8pt}

 We consider the compact type curve
 $C=C_1\cup C_2$, where $C_i$ are smooth curves of genus $i$ for $i=1,2$ and $C_1\cap C_2=\{x\}\, .$
We furthermore assume $x\in C_2$ is {\em not} a Weierstrass point, so
$C$ is not hyperelliptic. We will
determine the bitangents which have $C$ as the underlying curve.
Suppose
$$[C, p_1, p_2]\in \widetilde{\H}_3(2,2)\, $$
 so $p_1, p_2\in C\setminus \{x\}$.
We distinguish several cases:

\vspace{8pt}
\begin{enumerate}
\item[$\bullet$] $p_1, p_2\in C_1$\,  $\Rightarrow$\, $I(x,C_1)=3$ and $\omega_{C_2}=\OO_{C_2}(2x)$. Since $x$ is not a Weierstrass point, the case is impossible.

\vspace{6pt}
\item[$\bullet$] $p_1\in C_1$ and $p_2\in C_2$\,   $\Rightarrow$\,   $I(x,C_1)=1$.
We obtain the conditions
$$\OO_{C_1}(2p_1)=\OO_{C_1}(2x) \ \mbox{ and } \omega_{C_2}=\OO_{C_2}(2p_2).$$
All solutions are smoothable by Proposition \ref{sqqr2}.
The case accounts for $18$ points in the fibre of the
projection $\overline{\H}_3(2,2)\rightarrow \mm_3$ over $[C]$, and shows that the cuspidal curve $$\phi=\phi_{\omega_{C_2}(2x)}:C_2\rightarrow \mathbb P^2$$ has $6$ bitangents of type 1  joining the cusp $\phi(x)$ with one of the Weierstrass points of $C$. Each such line is counted with multiplicity $3$ corresponding to the nontrivial $2$-torsion points of $[C_1,x]\in \mm_{1,1}$.

\vspace{6pt}
\item[$\bullet$] $p_1, p_2\in C_2$\, $\Rightarrow$\, $I(x,C_2)=1$.
We obtain the condition
$$\OO_{C_2}(2p_1+2p_2)=\omega_{C_2}(2x).$$
The cohomology class of the images of both maps $f_i:C_2\rightarrow \mbox{Pic}^2(C_2)$ given by
$$f_1(y)=\OO_{C_2}(2y), \ \ f_2(y)=\omega_{C_2}(2x-2y)$$
equals $4\theta \in H^2(\mbox{Pic}^2(C_2), \mathbb Z)$, hence
$$[f_1(C_2)]\cdot [f_2(C_2)]=16\theta^2=32\, .$$
The cases where $p_1=x$ and $p_2\in C_2$ is a Weierstrass point
must be discarded, so we are left with $10=16-6$ possibilities for
the bitangent lines.
\end{enumerate}

\vskip 8pt

Suppose one of the points $p_1$ or $p_2$ tends to  $x$.
The case when both $p_1$ and $p_2$ merge to $x$
leads to $x\in C_2$ being a Weierstrass point and can be discarded.
We are thus left with the following situation.

\vspace{8pt}
\begin{enumerate}
\item[$\bullet$] $[C_1\cup E \cup C_2, p_1, p_2]\in \widetilde{\H}_3(2,2)$
with $p_1\in E\cong \mathbb P^1$ and $p_2\in C_2$. Let
 $$\{y\}=C_1\cap E \ \ \ \text{and} \ \ \  \{x\}=C_2\cap E\, .$$
We compute $I(y,C_1)=-1$ and $I(x,C_2)=-1$, therefore $x$ is a Weierstrass point of $C_2$.

The locus of such points forms a $5$-dimensional component of $\widetilde{\H}_3(2,2)$ that lies entirely in the boundary $\partial \mm_{3,2}$
(as can be checked by residue restrictions or limit
linear series arguments). The general point can {\em not} be smoothed.
\end{enumerate}
\vspace{8pt}

We have verified the line
corresponding to
$(0,1)$ in Table 1: $b_0=10$, $b_1=6$, $b_2=0$, and
$28= 10+ 3\cdot 6\, .$


\section{Twisted $k$-canonical divisors}\label{nndd}
\subsection{Moduli space}
Let $g$ and $n$ be in the stable range $2g-2+n>0$, and let
$k\in \mathbb{Z}_{\geq 0}$.
Let $$\mu=(m_1,\ldots,m_{n})\, , \ \ \ m_i\in \mathbb{Z}\, $$
be a vector satisfying
$$\sum_{i=1}^n m_i=k(2g-2)\,. $$
We define the closed substack  of {\em $k$-canonical divisors}
$\H^k_g(\mu)\subset \cM_{g,n}$ by
$$\H^k_g(\mu)=\Bigl\{[C, p_1, \ldots, p_n]\in \cM_{g,n}\, \Big| \, \OO_C\Bigl(\sum_{i=1}^n m_i p_i\Bigr)=\omega^k_C \Bigr\}\, . $$
The main focus of the paper has been on {\em differentials} which is the
$k=1$ case. {\em Rational functions} and {\em quadratic differentials} correspond
to $k=0$ and $k=2$ respectively.

\vspace{5pt}
\begin{definition} \label{ffgg} The divisor $\sum_{i=1}^n m_i p_i$
associated to $[C,p_1,\ldots,p_n]\in \mm_{g,n}$ is \newline
{\em $k$-twisted canonical} if there exists a twist
$I$ for which
$$\nu^*\OO_{C}\left(\sum_{i=1}^n m_ip_i\right)\, \stackrel{\sim}{=}\,
\nu^*(\omega^k_C)\otimes \OO_{C_I}
\left(\,\sum_{q\in \mathsf{N}_I} I(q,D'_q)\cdot q'+I(q, D''_q)\cdot q''\right)\,  $$
on the partial normalization $C_I$.
\end{definition}
\vspace{5pt}

We define the subset $\widetilde{\H}^k_g(\mu)\subset \mm_{g,n}$
parameterizing $k$-twisted canonical divisors by
$$\widetilde{\H}^k_g(\mu)=\left\{[C, p_1, \ldots, p_n]\in \mm_{g,n}\,
\Big| \,  \sum_{i=1}^n m_ip_i \ \ \text{is a $k$-twisted canonical divisor}\
\right\}\, .$$
By definition, we have
$$\widetilde{\H}^k_g(\mu)\cap \cM_{g,n} =\H^k_g(\mu)\, ,$$
 so $\H^k_g(\mu) \subset \widetilde{\H}^k_g(\mu)$ is an
open set.

\begin{theorem}\label{maink}
The moduli space
$\widetilde{\H}^k_g(\mu)\subset \mm_{g,n}$ is a closed substack
with irreducible components all of dimension at least $2g-3+n$.
\end{theorem}

\begin{proof} We
replace $\omega_C$ by $\omega_C^k$ in the proofs
of Proposition \ref{xzs}, Lemma \ref{gcct}, Lemma \ref{fffff},
and Proposition \ref{dim2}. The tensor power $k$
plays no essential role in the arguments.
\end{proof}

\subsection{Rational functions}
In the $k=0$ case of rational functions, the moduli space
$$\widetilde{\H}^0_g(\mu)\subset \mm_{g,n}$$ may have irreducible
components of various dimensions. No result parallel
to the  simple
dimension behavior of Theorems \ref{main1} and \ref{main2}
holds.

Let $\mu=(m_1,\ldots,m_n)$ be a vector satisfying
$$\sum_{i=1}^nm_i =0\,. $$
Let $\mm_{g}(\mathbb{P}^1,\mu)\,\widetilde{}\,$
be the moduli space of stable
maps to {\em rubber}{\footnote{See \cite{FaP} for definitions
and further references.}} with ramification profiles
over $0$ and $\infty$ determined by the positive and negative
part of $\mu$.
The following $k=0$
result has a straightforward proof which we leave to the reader.

\begin{proposition}
The image of the canonical map
$ \mm_{g}(\mathbb{P}^1,\mu)\ \widetilde{} \ \rightarrow\ \mm_{g,n} $
equals $$\widetilde{\H}_g^0(\mu)\subset \mm_{g,n}\, .$$
\end{proposition}

\subsection{Higher $k$}
For $k>1$, we do not know uniform dimension results for
$\H^k_g(\mu)$. Is the locus
$$\H^k_g(\mu)\subset {\mathcal{M}}_{g,n}$$
pure of expected codimension $g$?
Unfortunately, the answer is {\em no}:
if all parts of $\mu$ are non-negative
and divisible by $k$, then the sublocus
\begin{equation}\label{iss}
\H^1_g\left(\frac{1}{k}\cdot \mu\right) \subset \H^k_g(\mu)
\end{equation}
is of codimension $g-1$ in ${\mathcal{M}}_{g,n}$.

\vspace{8pt}
\noindent{\bf Question:}
{\em Perhaps  construction \eqref{iss} is the only
source of impure dimension?}
\vspace{8pt}

In the $k=2$ case of quadratic differentials, the dimension theory developed in
 \cite{Masur,Veech} answers the above question in the affirmative when
 $$ m_i\in \{-1,0,1,2,\ldots\} \ \ \text{for all $1\leq i \leq n$}\, .$$
 An affirmative answer for all $g$, $\mu$, and $k>1$ would perhaps be expected.

\newpage

\begin{center}
\bf{Appendix: The weighted fundamental class of $\widetilde{\H}_g(\mu)$} \\
\vspace{8pt}
{\em by F. Janda, R. Pandharipande, A. Pixton, and D. Zvonkine}
\end{center}

\vspace{10pt}
\noindent{A.1\, {\bf Stable graphs and strata.}}
\vspace{8pt}

In the strictly meromorphic case, the fundamental class of $\widetilde{\H}_g(\mu)$
weighted by intrinsic multiplicities is conjectured here to
be an explicit cycle in the tautological ring $R^*(\oM_{g,n})$
found by Pixton in 2014.
The formula is written in term of a summation
over stable graphs $\Gamma$ indexing
the strata of $\oM_{g,n}$.
The summand in Pixton's formula corresponding to $\Gamma$  is a
product over vertex, marking, and edge factors.
We review here the standard indexing of the strata
of $\oM_{g,n}$ by stable graphs.

The strata of the moduli space of curves correspond
to {\em stable graphs}
$$\Gamma=(\VV, \HH,\LL, \ \mathrm{g}:\VV \rarr \Z_{\geq 0},
\ v:\HH\rarr \VV,
\ \iota : \HH\rarr \HH)$$
satisfying the following properties:
\begin{enumerate}
\item[(i)] $\VV$ is a vertex set with a genus function $\g:\VV\to \Z_{\geq 0}$,
\item[(ii)] $\HH$ is a half-edge set equipped with a
vertex assignment $v:\HH \to \VV$ and an involution $\iota$,
\item[(iii)] $\EE$, the edge set, is defined by the
2-cycles of $\iota$ in $\HH$ (self-edges at vertices
are permitted),
\item[(iv)] $\LL$, the set of legs, is defined by the fixed points of $\iota$ and is
endowed with a bijective correspondence with a set of markings,
\item[(v)] the pair $(\VV,\EE)$ defines a {\em connected} graph,
\item[(vi)] for each vertex $v$, the stability condition holds:
$$2\g(v)-2+ \n(v) >0,$$
where $\n(v)$ is the valence of $\Gamma$ at $v$ including
both half-edges and legs.
\end{enumerate}
An automorphism of $\Gamma$ consists of automorphisms
of the sets $\VV$ and $\HH$ which leave invariant the
structures $\LL$, $\mathrm{g}$, $v$, and $\iota$.
Let $\text{Aut}(\Gamma)$ denote the automorphism group of $\Gamma$.

The genus of a stable graph $\Gamma$ is defined by:
$$\g(\Gamma)= \sum_{v\in V} \g(v) + h^1(\Gamma).$$
A stratum of the moduli space $\oM_{g,n}$
of Deligne-Mumford stable curves naturally determines
a stable graph of genus $g$ with $n$ legs by considering the dual
graph of a generic pointed curve parametrized by the stratum.
Let
$\mathsf{G}_{g,n}$
be the set of isomorphism classes of stable graphs of genus $g$ with $n$ legs.
The set $\mathsf{G}_{g,n}$ is finite.

To each stable graph $\Gamma$, we associate the moduli space
\begin{equation*}
\oM_\Gamma =\prod_{v\in \VV} \oM_{\g(v),\n(v)}.
\end{equation*}
 Let $\pi_v$ denote the projection from $\oM_\Gamma$ to
$\oM_{\g(v),\n(v)}$ associated to the vertex~$v$.  There is a
canonical
morphism
\begin{equation}\label{dwwd}
\xi_{\Gamma}: \oM_{\Gamma} \rarr \oM_{g,n}
\end{equation}
 with image{\footnote{
The degree of $\xi_\Gamma$ is $|\text{Aut}(\Gamma)|$.}}
equal to the  stratum
associated to the graph $\Gamma$.  To construct $\xi_\Gamma$,
a family of stable pointed curves over $\oM_\Gamma$ is required.  Such a family
is easily defined
by attaching the pull-backs of the universal families over each of the
$\oM_{\g(v),\n(v)}$  along the sections corresponding to half-edges.

 \vspace{10pt}
\noindent{A.2\, \bf{Additive generators of the tautological ring.}}
\vspace{8pt}

Let $\Gamma$ be a stable graph.
A {\em basic class} on $\oM_\Gamma$
is defined to be a product of monomials in $\kappa$ classes{\footnote{Our
convention is
$\kappa_i= \pi_*(\psi_{n+1}^{i+1})\ \in A^{i}(\oM_{g,n},\mathbb{Q})$ where
$$\pi: \oM_{g,n+1} \rightarrow \oM_{g,n}$$ is the
map forgetting the marking $n+1$. For a review of $\kappa$ and
and the cotangent $\psi$ classes, see \cite{GraPan}.}}
 at each
vertex of the graph and powers of $\psi$ classes at each half-edge (including the legs),
$$
\gamma = \prod_{v\in \mathrm{V}}
\prod_{i>0}\kappa_i[v]^{x_i[v]} \ \cdot
\ \prod_{h \in \mathrm{H}} \psi_{h}^{y[h]}\
\in A^*(\oM_\Gamma)\ ,
$$
where $\kappa_i[v]$ is the $i^{\rm th}$ kappa class on $\oM_{\mathrm{g}(v),\mathrm{n}(v)}$.
We impose the condition
$$
\sum_{i>0} i x_i[v]
+ \sum_{h\in \HH[v]} y[h] \leq
\text{dim}_{\mathbb{C}}\ \oM_{\mathrm{g}(v), \mathrm{n}(v)} =
3\mathrm{g}(v)-3+ \mathrm{n}(v)
$$
at each vertex to avoid  the trivial vanishing of $\gamma$.
Here,
$\HH[v]\subset \HH$
is the set of legs (including the half-edges) incident to~$v$.

Consider the $\mathbb{Q}$-vector space $\cS_{g,n}$ whose basis is given by the isomorphism classes of pairs $[\Gamma,\gamma]$, where $\Gamma$ is a stable graph of genus $g$ with $n$
legs and $\gamma$ is a basic class on $\oM_\Gamma$.
Since there are only finitely many pairs $[\Gamma, \gamma]$ up to isomorphism, $\cS_{g,n}$ is finite dimensional.
The canonical map
$$q:\cS_{g,n} \rightarrow R^*(\oM_{g,n})$$
is surjective \cite{GraPan} and provides additive generators of the tautological ring.
The kernel of $q$ is the ideal of tautological relations, see \cite[Section 0.3]{PPZ}.

\vspace{10pt}
\noindent{A.3\, \bf{Pixton's formula.}}
\vspace{8pt}

\noindent{A.3.1\, \bf{Weighting.}} Let $\mu=(m_1,\ldots,m_n)$ be a vector of zero and pole multiplicities
satisyfing
$$\sum_{i=1}^n m_i= 2g-2\, .$$
It will be convenient to work with the shifted vector defined by
$$\widetilde{\mu}=(m_1+1,\ldots, m_n+1)\, , \ \ \ \ \widetilde{m}_i=m_i+1\, .$$
Let
$\Gamma$ be a stable graph of genus $g$ with $n$ legs.
An {\em admissible weighting}
of $\Gamma$ is a function on the set of half-edges,
$$ w:\HH(\Gamma) \rightarrow \mathbb{Z},$$
which satisfies the following three properties:
\begin{enumerate}
\item[(i)] $\forall h_i\in \LL(\Gamma)$, corresponding to
 the marking $i\in \{1,\ldots, n\}$,
$$w(h_i)=\widetilde{m}_i\ ,$$
\item[(ii)] $\forall e \in \EE(\Gamma)$, corresponding to two half-edges
$h(e),h'(e) \in \HH(\Gamma)$,
$$w(h)+w(h')=0\,$$
\item[(iii)] $\forall v\in \VV(\Gamma)$,
$$\sum_{h\mapsto v} w(h)=2 \g(v)-2+\n(v)\, ,$$
where the sum is taken over {\em all} $\mathsf{n}(v)$ half-edges incident
to $v$.
\end{enumerate}

Let $r$ be a positive integer.
An {\em admissible weighting mod $r$}
of $\Gamma$ is a function,
$$ w:\HH(\Gamma) \rightarrow \{0,\ldots, r-1\},$$
which satisfies exactly properties (i-iii) above, but
with the equalities replaced, in each case, by the condition of
{\em congruence $mod$ $r$}.
For example, for (i), we require
$$w(h_i)=\widetilde{m}_i \mod r\, .$$
Let $\mathsf{W}_{\Gamma,r}$ be the set of admissible weightings mod $r$
of $\Gamma$. The set $\mathsf{W}_{\Gamma,r}$ is finite.

\vspace{8pt}
\noindent{A.3.2\, \bf{Pixton's cycle.}}
Let $r$ be a positive
integer.
We denote by
$\PPP_{g,\mu}^{d,r}\in R^d(\oM_{g,n})$ the degree $d$ component of the tautological class
\begin{multline*}
\hspace{-10pt}\sum_{\Gamma\in \mathsf{G}_{g,n}}
\sum_{w\in \mathsf{W}_{\Gamma,r}}
\frac1{|{\text{Aut}}(\Gamma)| }
\,
\frac1{r^{h^1(\Gamma)}}
\;
\xi_{\Gamma*}\Bigg[
\prod_{v\in \VV(\Gamma)} \exp(-\kappa_1[v]) \,
\prod_{i=1}^n \exp(\widetilde{m}_i^2 \psi_{h_i}) \cdot
\\ \hspace{+10pt}
\prod_{e=(h,h')\in \EE(\Gamma)}
\frac{1-\exp(-w(h)w(h')(\psi_h+\psi_{h'}))}{\psi_h + \psi_{h'}} \Bigg]\, .
\end{multline*}

The following fundamental polynomiality property of $\PPP_{g,\mu}^{d,r}$  has
been proven by Pixton.

\vspace{9pt}
\noindent {\bf{Proposition}} (Pixton \cite{Pix}). {\em For fixed $g$, $\mu$,
 and $d$, the \label{pply}
class
$$\PPP_{g,\mu}^{d,r} \in R^d(\oM_{g,n})$$
is polynomial in $r$ for sufficiently large $r$.}
\vspace{9pt}

We denote by $\PPP_{g,\mu}^d$ the value at $r=0$
of the polynomial associated to $\PPP_{g,\mu}^{d,r}$ by Proposition \ref{pply}. In other words,
$\PPP_{g,\mu}^d$ is the {\em constant} term of the associated polynomial in $r$.
The cycle $\PPP_{g,\mu}^d$ was constructed by Pixton \cite{Pix} in 2014.

If $d>g$, Pixton conjectured (and Clader and Janda have proven
\cite{jc}) the vanishing
$$\PPP_{g,\mu}^d =0 \ \in R^*(\overline{\mathcal{M}}_{g,n})\, .$$
If $d=g$, the class $\PPP_{g,\mu}^g$ is nontrivial. When
restricted to the moduli of curves of compact type,
$\PPP_{g,\mu}^g$ is related to canonical divisors
via the Abel-Jacobi map by earlier work of Hain \cite{rh}.
We propose a precise relationship between $\PPP_{g,\mu}^g$
and a  weighted fundamental classes of the moduli space of
twisted canonical divisors on $\overline{\mathcal{M}}_{g,n}$.

\vspace{10pt}
\noindent{A.4\, \bf{The moduli of twisted canonical divisors.}}
\vspace{8pt}

We consider here the strictly meromorphic case where
$\mu=(m_1,\ldots,m_n)$ has at least one negative part.
We construct a cycle
$$\mathsf{H}_{g,\mu} \in A^g(\oM_{g,n})$$
by summing over the irreducible components of
the moduli space $\widetilde{\H}_g(\mu)$
of twisted canonical divisors defined in the paper.

Let $\mathsf{S}_{g,\mu}\subset \mathsf{G}_{g,n}$ be the set of {\em simple} star
graphs defined by the following properties:

\begin{enumerate}
\item[$\bullet$]
 $\Gamma\in \mathsf{S}_{g,\mu}$ has
a single center vertex $v_0$,
\item[$\bullet$]
edges (possibly multiple) of $\Gamma$ connect
$v_0$ to
outlying vertices $v_1,\ldots, v_k$,
\item[$\bullet$]
the negative parts of $\mu$ are distributed to the center $v_0$ of $\Gamma$.
\end{enumerate}
A simple star graph $\Gamma$ has no self-edges. Let
$\VV_{\mathrm{out}}(\Gamma)$ denote the set of outlying vertices.
The simplest
star graph consists of the center alone with no outlying vertices.

A twist $I$ of a simple star graph is a function
$$I: \mathsf{E}(\Gamma) \rightarrow \mathbb{Z}_{>0}\ $$
satisfying:
\begin{enumerate}
\item[$\bullet$] for the center $v_0$ of $\Gamma$,
$$2\g(v_0)-2  +\sum_{e\mapsto v_0} (I(e)+1)= \sum_{i\mapsto v_0} m_i \, ,$$
\item[$\bullet$] for each outlying  vertex $v_j$ of $\Gamma$,
$$2\g(v_j)-2  +\sum_{e\mapsto v_j} (-I(e)+1)= \sum_{i\mapsto v_j} m_i \, ,$$
\end{enumerate}
Let $\mathrm{Tw}(\Gamma)$ denote the finite set of possible twists.

The notation $i\mapsto v$ in  above equations
denotes markings (corresponding to
the parts of $\mu$) which are incident to the vertex $v$.
We denote by $\mu[v]$ the vector consisting of the parts of
$\mu$ incident to $v$.
Similarly,
$e\mapsto v$ denotes edges incident to the vertex $v$.
We denote by $-I[v_0]-1$ the vector of values $-I(e)-1$ indexed by all edges
incident to $v_0$, and by
$I[v_j]-1$ the vector of values $I(e)-1$ indexed by all edges
incident to an outlying vertex $v_j$.

We define the weighted fundamental class
$\mathsf{H}_{g,\mu}\in A^g(\oM_{g,n})$  of the moduli space
$\widetilde{H}_g(\mu)$ of twisted canonical divisors in the
strictly meromorphic case by
\begin{multline*}
\mathsf{H}_{g,\mu} =
\sum_{\Gamma\in \mathsf{S}_{g,\mu}}
\sum_{I\in {\mathrm{Tw}}(\Gamma)}
\frac {\prod_{e\in \mathsf{E}(\Gamma)} I(e)}{|{\text{Aut}}(\Gamma)| }
\;
\xi_{\Gamma*}\Bigg[
\Big[\overline{\H}_{\g(v_0)}
\big(\mu[v_0],-I[v_0]-1\big)\Big]
\\
\cdot \prod_{v\in \VV_{\mathrm{out}}(\Gamma)}
\Big[\overline{\H}_{\g(v)}\big(\mu[v],I[v]-1\big)\Big]\Bigg] \, .
\end{multline*}

The right side of the definition of $\mathsf{H}_{g,\mu}$ may be viewed
as a sum over all irreducible components of
$$Z\subset\widetilde{\H}_{g}(\mu)\, .$$
If the curves of $Z$ generically do not have a
separating node,
$$Z\subset \overline{\H}_g^{\mathrm{Irr}}(\mu)
\subset
\overline{\mathcal{M}}_{g,n}\, ,$$
see Section \ref{irrr}.
Since there is an equality of closures $\overline{\H}_g(\mu)=
\overline{\H}_g^{\mathrm{Irr}}(\mu)$,
we obtain
$$Z\subset \overline{\H}_g(\mu) \subset
\overline{\mathcal{M}}_{g,n}\, .$$
Hence,
$Z$ contributes to the term corresponding
to the trivial star graph $$\Gamma=\{v_0\}$$
 of genus $g$ carrying
all the parts of $\mu$.

The trivial star graph $\Gamma$ has
no edges, nothing to twist, and no automorphisms.
Since $\xi_{\Gamma}$ is the identity map here, the
term corresponding to $\Gamma$ is
$$
\Big[\overline{\H}_{\g(v_0)}
\big(\mu)\Big]
\in A^{\g(v_0)}(\overline{\mathcal{M}}_{\g(v_0),\n(v_0)})\, .$$

If $Z$ has a separating node, then, by Section \ref{stt} of the paper,
$Z$ contributes to the term corresponding to the associated
simple star graph.
 Every outlying vertex contributes the class
$$\Big[\overline{\H}_{\g(v_j)}\big(\mu[v_j],I[v_j]-1\big)\Big] \in
A^{\g(v_j)-1}(\overline{\mathcal{M}}_{\g(v_j),\n(v_j)})$$
corresponding to a canonical divisor in the holomorphic
case (since the parts of $\mu[v_j]$ and $I[v_j]-1$ are non-negative).

The formula for $\mathsf{H}_{g,\mu}$ differs from the usual
fundamental class of $\widetilde{\H}_g(\mu)$ by the weighting
factor ${\prod_{e\in \mathsf{E}(\Gamma)} I(e)}$ which is
motivated by relative Gromov-Witten theory.

\vspace{10pt}
\noindent{\bf{Conjecture A.}} {\em In the strictly meromorphic case,
$\mathsf{H}_{g,\mu} = 2^{-g}\, \mathsf{P}_{g,\mu}^g $ .}
\vspace{10pt}

Our conjecture provides a completely geometric representative
of Pixton's cycle in terms of twisted canonical divisors.
The geometric situation here is parallel, but much better behaved,
than the corresponding result of \cite{jppz} proving
Pixton's conjecture for the double ramification cycle (as
the latter carries virtual contributions from contracted
components of excess dimension).

Finally, we speculate that the study of volumes of the
moduli spaces of meromorphic differentials may have a much
simpler answer if summed over all the components
of $\widetilde{\H}_g(\mu)$. How to properly define such volumes
here is a question left for the future.

\vspace{10pt}
\noindent{A.5\, \bf{Examples.}}
\vspace{8pt}

\noindent{A.5.1\, \bf{Genus at most 1.}}
Conjecture A is trivial in case $g=0$ since both sides are the
identity in $R^0(\overline{\mathcal{M}}_{0,n})$.
In case $g=1$, Conjecture A is nontrivial and true.
The class $\mathsf{H}_{1,\mu}$
has contributions from boundary components.
Our of proof Conjecture A in case $g=1$ is indirect. We prove
\begin{enumerate}
\item[$\bullet$] $\mathsf{H}_{1,\mu}$
equals the double ramification cycle in genus 1 associated to $\mu$
 via geometric arguments.
\item[$\bullet$] $2^{-1}\PPP^1_{1,\mu}$ equals
Pixton's prediction for the double ramification cycle via
direct analysis of the formulas.
\end{enumerate}
Then, by the main result of \cite{jppz}, Conjecture A
holds in $g=1$.
An alternate approach via direct calculation
is also available in $g=1$.

\vspace{8pt}
\noindent{A.5.2\, \bf{Genus 2.}}
A more interesting example is  $g=2$ with $\mu=(3,-1)$.
We first enumerate all simple star
graphs (together with their possible twists) which contribute to
the cycle $\mathsf{H}_{2,(3,-1)}$:

\begin{enumerate}
\item[(i)] $\Gamma=\{v_0\}$,\, $\g(v_0)=2$,\, $|\text{Aut}(\Gamma)|=1$.

\begin{centering}{\hspace{50pt} \includegraphics[scale=0.5]{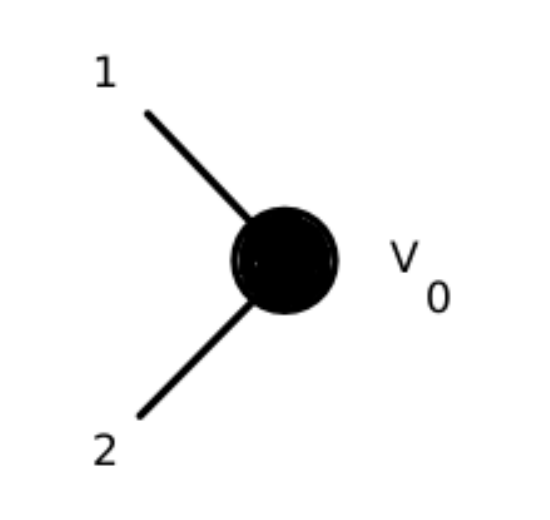}}
\end{centering}

\item [(ii)] $\Gamma=\{v_0,v_1\}$,\, $\g(v_0)=1$,\, $\g(v_1)=1$,\, $|\text{Aut}(\Gamma)|=1$.

\begin{centering}{\hspace{50pt} \includegraphics[scale=0.5]{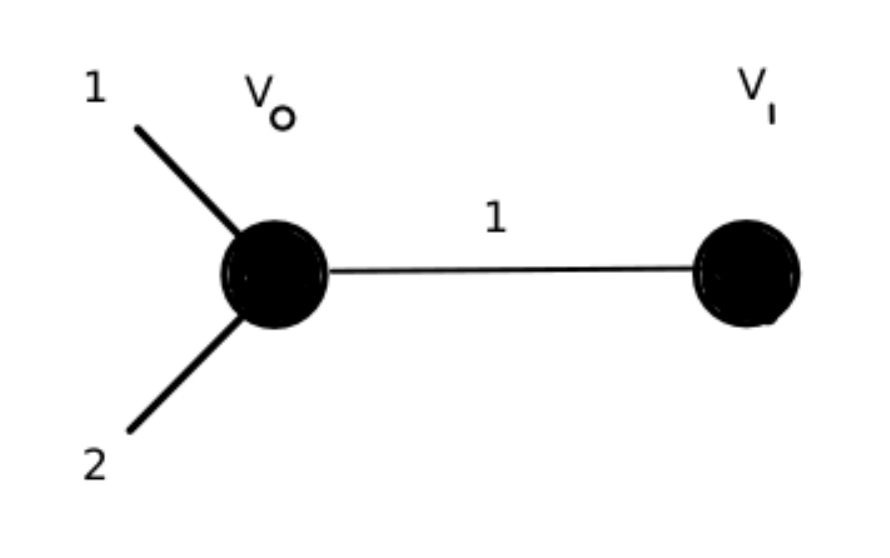}}
\end{centering}

\item[ (iii)]  $\Gamma=\{v_0,v_1\}$,\, $\g(v_0)=0$,\, $\g(v_1)=2$,\, $|\text{Aut}(\Gamma)|=1$.

\begin{centering}{\hspace{50pt} \includegraphics[scale=0.5]{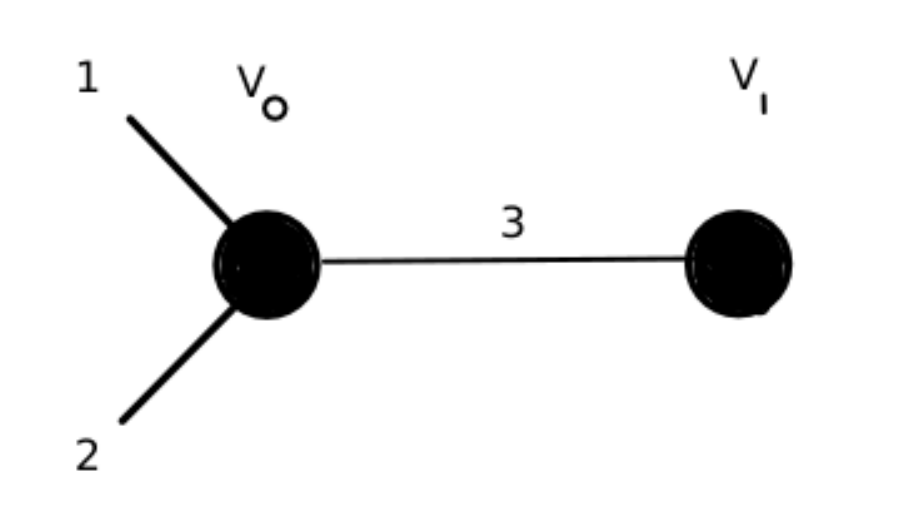}}
\end{centering}

\item[ (iv)]  $\Gamma=\{v_0,v_1,v_2\}$,\, $\g(v_0)=0$,\, $\g(v_1)=1$,\, $\g(v_2)=1$,\, $|\text{Aut}(\Gamma)|=2$.

\begin{centering}{\hspace{40pt} \includegraphics[scale=0.5]{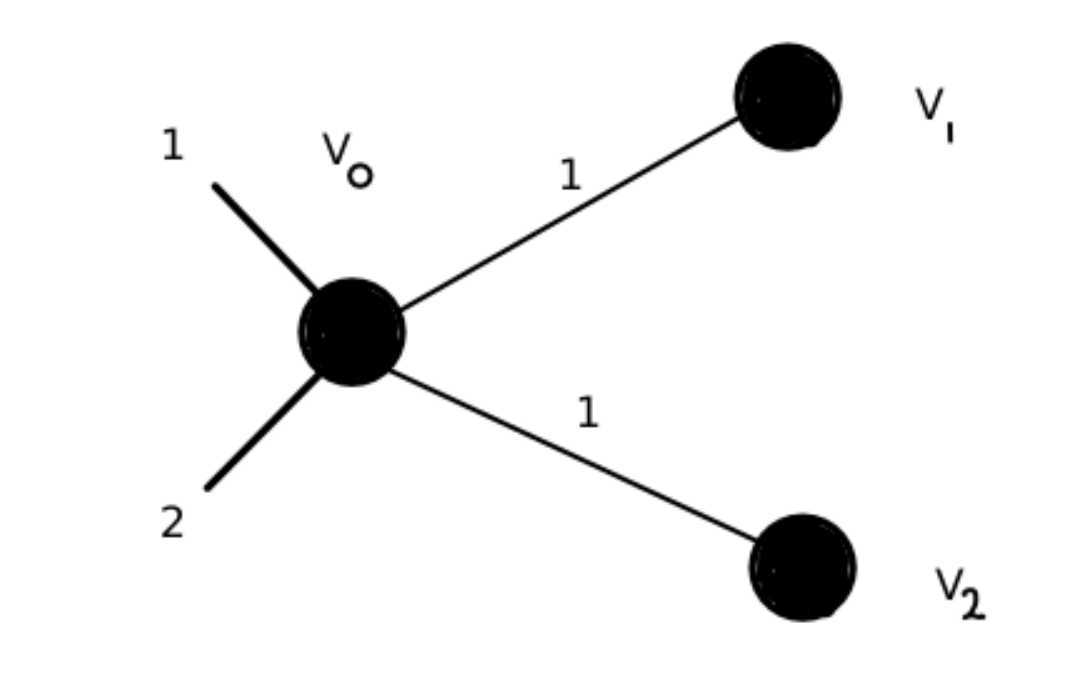}}
\end{centering}

\item[ (v)]  $\Gamma=\{v_0,v_1\}$,\, $\g(v_0)=0$,\, $\g(v_1)=1$, \, $|\text{Aut}(\Gamma)|=2$.

\vspace{4pt}
\begin{centering}{\hspace{50pt} \includegraphics[scale=0.5]{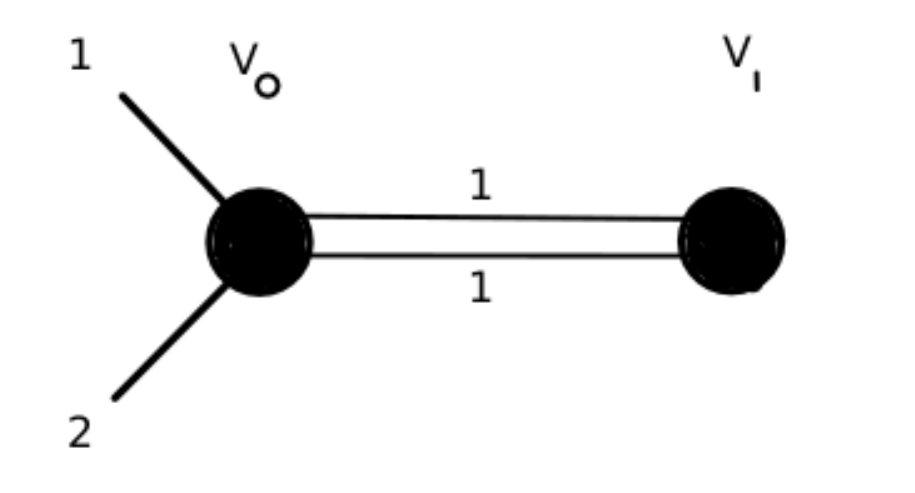}}
\end{centering}

\end{enumerate}
In the diagrams, the legs are assigned a marking, and the edges are assigned a twist.
In all the cases $\Gamma$ above, the twist $I\in \text{Tw}(\Gamma)$ is unique.

The second step is to calculate, for each star graph $\Gamma$ and associated twist $I\in \text{Tw}(\Gamma)$,
the contribution
\begin{equation*}
\frac {\prod_{e\in \mathsf{E}(\Gamma)} I(e)}{|{\text{Aut}}(\Gamma)| }
\;
\xi_{\Gamma*}\Bigg[
\Big[\overline{\H}_{\g(v_0)}
\big(\mu[v_0],-I[v_0]-1\big)\Big]
\cdot \prod_{v\in \VV_{\mathrm{out}}(\Gamma)}
\Big[\overline{\H}_{\g(v)}\big(\mu[v],I[v]-1\big)\Big]\Bigg] \, .
\end{equation*}
The contributions in cases (i)-(v) above are:
\begin{enumerate}
\item[(i)] The moduli space $\overline{\H}_{2}\big(3,-1\big)$ is
empty since there are no meromorophic differentials with a simple pole. The contribution is $0$.

\vspace{5pt}
\item [(ii)] We must calculate here $\overline{\H}_{1}\big(3,-1,-2\big)$
which is easily obtained from the method of test curves (or by applying Conjecture A in
the proven genus $1$ case). The  contribution is

\begin{center} \begin{picture}(0,0)%
\includegraphics{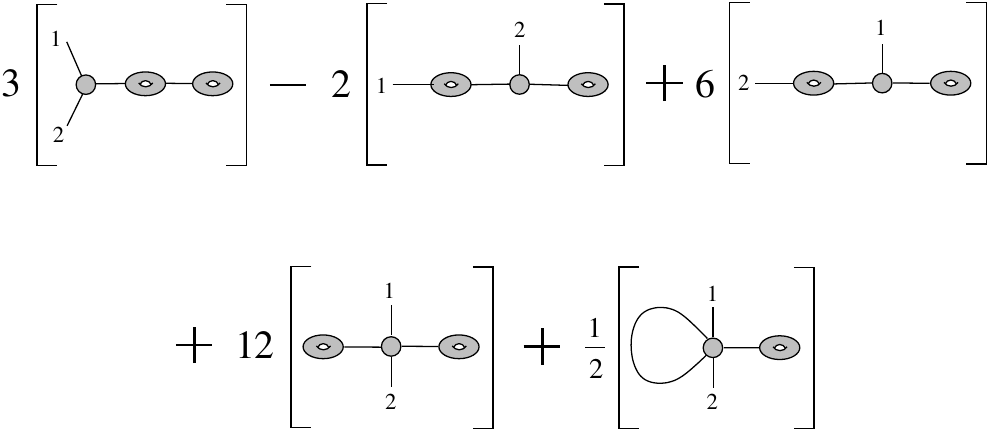}%
\end{picture}%
%
%
\setlength{\unitlength}{1618sp}%
\begingroup\makeatletter\ifx\SetFigFont\undefined%
\gdef\SetFigFont#1#2#3#4#5{%
  \reset@font\fontsize{#1}{#2pt}%
  \fontfamily{#3}\fontseries{#4}\fontshape{#5}%
  \selectfont}%
\fi\endgroup%
\begin{picture}(11572,5034)(3977,-5815)
\end{picture}%
.
\end{center}

\vspace{5pt}
\item[(iii)] We require here the well-known formula for
the Weierstrass locus,
$$\overline{\H}_{2}\big(2\big)\subset \overline{\mathcal{M}}_{2,1}\, ,$$  studied
by Eisenbud and Harris \cite{EH}. Lemma 5 of \cite{BP} follows our notation here.
The contribution, including the
twist $3$, is

\begin{center} \begin{picture}(0,0)%
\includegraphics{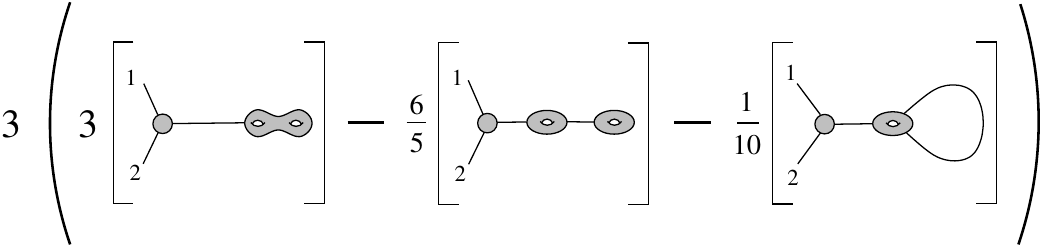}%
\end{picture}%
%
%
\setlength{\unitlength}{1618sp}%
\begingroup\makeatletter\ifx\SetFigFont\undefined%
\gdef\SetFigFont#1#2#3#4#5{%
  \reset@font\fontsize{#1}{#2pt}%
  \fontfamily{#3}\fontseries{#4}\fontshape{#5}%
  \selectfont}%
\fi\endgroup%
\begin{picture}(12184,2881)(1306,-3200)
\put(3885,-2037){\makebox(0,0)[lb]{\smash{{\SetFigFont{8}{9.6}{\rmdefault}{\mddefault}{\updefault}{\color[rgb]{0,0,0}$\psi$}%
}}}}
\end{picture}%
 .
\end{center}

\vspace{5pt}
\item[(iv)] The locus is already codimension $2$, so the contribution (including the
automorphism factor) is

\begin{center} \begin{picture}(0,0)%
\includegraphics{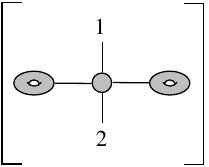}%
\end{picture}%
%
%
\setlength{\unitlength}{1618sp}%
\begingroup\makeatletter\ifx\SetFigFont\undefined%
\gdef\SetFigFont#1#2#3#4#5{%
  \reset@font\fontsize{#1}{#2pt}%
  \fontfamily{#3}\fontseries{#4}\fontshape{#5}%
  \selectfont}%
\fi\endgroup%
\begin{picture}(2407,1940)(1436,-5814)
\end{picture}%
 .
\end{center}

\vspace{5pt}
\item[(v)] The locus is again codimension $2$, so the contribution (including the
automorphism factor) is

\begin{center} \begin{picture}(0,0)%
\includegraphics{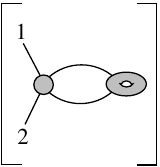}%
\end{picture}%
%
%
\setlength{\unitlength}{1618sp}%
\begingroup\makeatletter\ifx\SetFigFont\undefined%
\gdef\SetFigFont#1#2#3#4#5{%
  \reset@font\fontsize{#1}{#2pt}%
  \fontfamily{#3}\fontseries{#4}\fontshape{#5}%
  \selectfont}%
\fi\endgroup%
\begin{picture}(1858,1938)(9052,-5801)
\end{picture}%
 .
\end{center}

\end{enumerate}

After summing over the cases (i)-(v), we obtain a formula
in $R^2(\overline{\mathcal{M}}_{2,2})$
for the weighted fundamental class $\mathsf{H}_{2,(3,-1)}$
of the moduli space of twisted canonical divisors:
\vspace{5pt}

\begin{center} \begin{picture}(0,0)%
\includegraphics{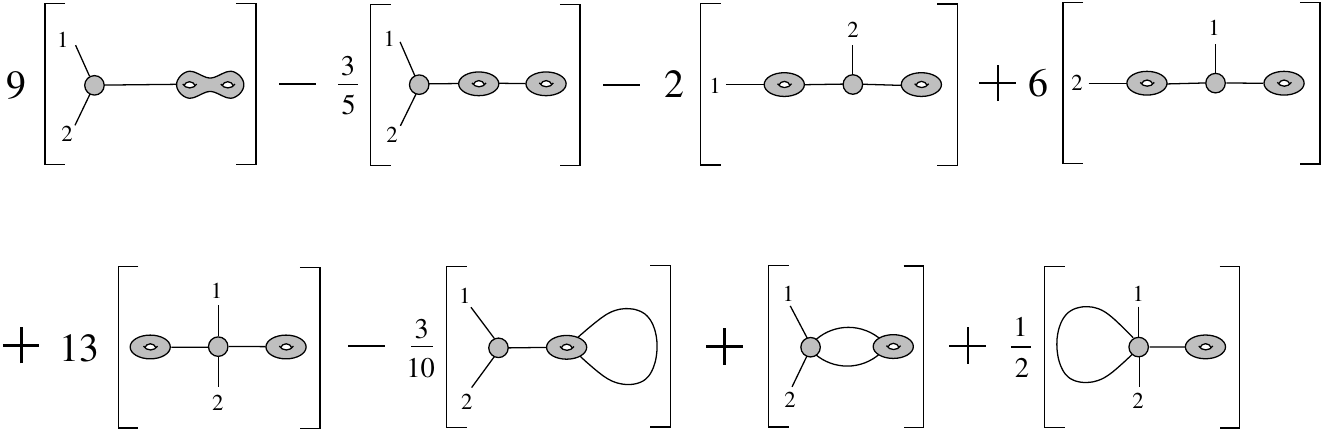}%
\end{picture}%
%
%
\setlength{\unitlength}{1618sp}%
\begingroup\makeatletter\ifx\SetFigFont\undefined%
\gdef\SetFigFont#1#2#3#4#5{%
  \reset@font\fontsize{#1}{#2pt}%
  \fontfamily{#3}\fontseries{#4}\fontshape{#5}%
  \selectfont}%
\fi\endgroup%
\begin{picture}(15474,5033)(75,-5814)
\put(1855,-2046){\makebox(0,0)[lb]{\smash{{\SetFigFont{8}{9.6}{\rmdefault}{\mddefault}{\updefault}{\color[rgb]{0,0,0}$\psi$}%
}}}}
\end{picture}%
 .
\end{center}


\vspace{8pt}
\noindent The result exactly equals $2^{-2}\, \PPP_{2,(3,-1)}^2 \in
R^2(\overline{\mathcal{M}}_{2,2})$.

The match is obtained
by expanding $2^{-2}\, \PPP_{2,(3,-1)}^2$ via the definition in Section A.3.2
and then applying relations in the tautological ring.
Conjecture A is therefore true for $g=2$ with $\mu=(3,-1)$.

We have also checked Conjecture A for $g=2$ with $\mu=(2,1,-1)$.
The calculation of $\mathsf{H}_{2,(2,1,-1)}$ involves 7 simple star
graphs,  Lemmas 5 and 6 of \cite{BP}, and 23 strata classes of
$R^2(\overline{\mathcal{M}}_{2,3})$. The matching with
$$2^{-2}\, \PPP_{2,(2,1,-1)}^2 \in R^2(\overline{\mathcal{M}}_{2,3})$$
as predicted by Conjecture A exactly holds.

\newpage
\noindent{A.6\, \bf{The cycle $\overline{\H}_g(\mu)$}}
\vspace{8pt}

\noindent{A.6.1\, \bf{Meromorphic differentials.}}
In the strictly meromorphic case, where
$$\mu=(m_1,\ldots,m_n)$$ has at least one negative part, Conjecture A provides
a formula in the tautological ring for the weighted fundamental class
$$\mathsf{H}_{g,\mu} \in A^g(\oM_{g,n})\, .$$
Can Conjecture A also be used to determine the class
$$[\overline{\H}_g(\mu)] \in A^g(\oM_{g,n})$$
of
closure of the moduli of canonical divisors
$\H_{g}(\mu)$ on nonsingular curves ?
We will prove that the answer is {\em yes}.

We have seen $[\overline{\H}_g(\mu)]$ appears exactly
as the
contribution to  $\mathsf{H}_{g,\mu}$ of the trivial simple star graph.
Rewriting the definition of $\mathsf{H}_{g,\mu}$ yields the
formula:

\begin{multline} \label{pppzzz}
[\overline{\H}_{g}(\mu)] = \mathsf{H}_{g,\mu} -
\sum_{\Gamma\in \mathsf{S}^\star_{g,\mu}}
\sum_{I\in {\mathrm{Tw}}(\Gamma)}
\frac {\prod_{e\in \mathsf{E}(\Gamma)} I(e)}{|{\text{Aut}}(\Gamma)| }
\;
\xi_{\Gamma*}\Bigg[
\Big[\overline{\H}_{\g(v_0)}
\big(\mu[v_0],-I[v_0]-1\big)\Big]
\\
\cdot \prod_{v\in \VV_{\mathrm{out}}(\Gamma)}
\Big[\overline{\H}_{\g(v)}\big(\mu[v],I[v]-1\big)\Big]\Bigg] \,
\end{multline}
where $\mathsf{S}^\star_{g,\mu}$ is the set of {\em nontrivial}
simple star graphs.

A nontrivial simple star graph ${\Gamma\in \mathsf{S}^\star_{g,\mu}}$
must have at least one outlying vertex.
An outlying vertex ${v\in \VV_{\mathrm{out}}(\Gamma)}$
contributes the factor
$$\Big[\overline{\H}_{\g(v)}\big(\mu[v],I[v]-1\big)\Big]$$
which concerns the {\em holomorphic} case. Hence if $\g(v)=0$,
the contribution of $\Gamma$ vanishes (as there are no
holomorphic differential in genus $0$).

To state the first induction result, we require the standard partial order
on the pairs $(g,n)$. We define
\begin{equation}\label{kk12}
(\widehat{g},\widehat{n}) \stackrel{\circ}{<} (g,n)
\end{equation}
if $\widehat{g}<g$ holds or if
$$\widehat{g}=g \ \ \ \text{and}\ \ \  \widehat{n} <n\, $$
both hold.

\vspace{8pt}
\noindent{\bf Lemma A1.}
{\em Let $\mu=(m_1,\ldots,m_n)$ be strictly meromorphic.
The class
$$[\overline{\H}_g(\mu)] \in A^g(\oM_{g,n})$$ is
determined via formula \eqref{pppzzz} by classes of the following three types:
\begin{enumerate}
\item[(i)] $\mathsf{H}_{g,\mu}\in A^g(\oM_{g,n})$,
\item[(ii)] $[\overline{\H}_{g'}(\mu')]\in A^{g'}(\oM_{g',n'})$ for
$(g',n') \stackrel{\circ}{<}(g,n)$ and $\mu'=(m'_1,\ldots,m'_{n'})$ strictly meromorphic,
\item[(iii)] $[\overline{\H}_{g''}(\mu'')]\in A^{g''-1}(\oM_{g'',n''})$
for $(g'',n'')\stackrel{\circ}{<} (g,n)$ and $\mu''=(m''_1,\ldots,m''_{n''})$
holomorphic.
\end{enumerate}}

\vspace{8pt}
\begin{proof}
The class (i) appears as the leading term on the right side of
\eqref{pppzzz}. Nontrivial simple star graphs $\Gamma$ contribute
classes of type (ii) via the center vertex $v_0$ (with $$g'=\g(v_0)<g$$
or else the
contribution of $\Gamma$ vanishes). The outlying vertices $v$ of
$\Gamma$ contribute classes of type (iii).
If $$g''=\g(v)=g\, ,$$ then
$\g(v_0)=0$ and there are no outlying vertices other than $v$
(or else the contribution of $\Gamma$ vanishes)
and only a single edge. By stability,
at
least two parts of $\mu$ must be distributed to $v_0$, so $n''<n$.
\end{proof}

 Lemma A.1 alone does {\em not} let us recursively
calculate $[\overline{\H}_g(\mu)]$ in terms of
$\mathsf{H}_{g,\mu}$ because of occurances of the holomorphic
cases (iii).

\vspace{8pt}
\noindent{A.6.2\, \bf{Holomorphic differentials.}}
Consider the holomorphic case where
$$\mu=(m_1,\ldots,m_n)$$
has only non-negative parts. Since
$$\sum_{i=1}^n m_i = 2g-2\, $$
we must have $g\geq 1$. If $g=1$, then all parts of $\mu$ must be
$0$ and $$[\overline{\H}_1(0,\ldots,0)] = 1 \in
A^0(\overline{\mathcal{M}}_{1,n})\,. $$
We assume $g\geq 2$.

Let $\mu'=(m'_1,\ldots,m'_{n'})$ be obtained from $\mu$ by removing the parts equal to $0$.
Then
$[\overline{\H}_g(\mu)]$
is obtained by pull-back{\footnote{The pull-back is true on
the level of subvarieties,
$\overline{\H}_g(\mu)=\tau^{-1}\Big(\overline{\H}_g(\mu')\Big)
\subset \overline{\mathcal{M}}_{g,n}$ .
}} via the map
$$\tau: \overline{\mathcal{M}}_{g,n}\rightarrow \overline{\mathcal{M}}_{g,n'}$$
forgetting the markings
associated to the
$0$ parts of $\mu$,
\begin{equation}\label{j123}
[\overline{\H}_g(\mu)]=\tau^*[\overline{\H}_g(\mu')]
\in A^{g-1}(\overline{\mathcal{M}}_{g,n})\, .
\end{equation}

We assume $\mu$ has no $0$ parts, so all parts $m_i$ are positive.
We place the parts of $\mu$ in increasing order
$$m_1\leq m_2 \leq \cdots \leq m_{n-1} \leq m_n\, , $$
so $m_n$ is the largest part.
Let $\mu^+$ be the partition defined by
$$\mu^+=(m_1,\ldots,m_{n-1},m_n+1,-1)\, .$$
We have increased the largest part of $\mu$, added
a negative part, and preserved the sum
$$|\mu| = |\mu^+|\, .$$
For notational convenience, we will  write
$$\mu^+ = (m_1^+, \ldots, m_n^+, -1)\,, \ \ \ m^+_{i<n}=m_i\, , \ \ \ m^+_n=m_n+1\, .$$
The length of $\mu^+$ is $n+1$.

Since $\mu^+$ is strictly meromorphic, we are permitted to
apply formula \eqref{pppzzz}. We obtain
\begin{multline*}
[\overline{\H}_{g}(\mu^+)] = \mathsf{H}_{g,\mu^+} -
\sum_{\Gamma\in \mathsf{S}^\star_{g,\mu^+}}
\sum_{I\in {\mathrm{Tw}}(\Gamma)}
\frac {\prod_{e\in \mathsf{E}(\Gamma)} I(e)}{|{\text{Aut}}(\Gamma)| }
\;
\xi_{\Gamma*}\Bigg[
\Big[\overline{\H}_{\g(v_0)}
\big(\mu^+[v_0],-I[v_0]-1\big)\Big]
\\
\cdot \prod_{v\in \VV_{\mathrm{out}}(\Gamma)}
\Big[\overline{\H}_{\g(v)}\big(\mu^+[v],I[v]-1\big)\Big]\Bigg] \,
\end{multline*}
where $\mathsf{S}^\star_{g,\mu^+}$ is the set of nontrivial
simple star graphs.

Since a meromorphic differential on a nonsigular curve
{\em can not} have just one simple pole, ${\H}_{g}(\mu^+)$
is empty and $[\overline{\H}_{g}(\mu^+)]=0$.
We rewrite the above equation as
\begin{multline} \label{tttzzz}
\mathsf{H}_{g,\mu^+} =
\sum_{\Gamma\in \mathsf{S}^\star_{g,\mu^+}}
\sum_{I\in {\mathrm{Tw}}(\Gamma)}
\frac {\prod_{e\in \mathsf{E}(\Gamma)} I(e)}{|{\text{Aut}}(\Gamma)| }
\;
\xi_{\Gamma*}\Bigg[
\Big[\overline{\H}_{\g(v_0)}
\big(\mu^+[v_0],-I[v_0]-1\big)\Big]
\\
\cdot \prod_{v\in \VV_{\mathrm{out}}(\Gamma)}
\Big[\overline{\H}_{\g(v)}\big(\mu^+[v],I[v]-1\big)\Big]\Bigg] \, .
\end{multline}

A nonvanishing term on right side of \eqref{tttzzz} corresponding to
${\Gamma\in \mathsf{S}^\star_{g,\mu^+}}$
has a center vertex factor
$$ \Big[\overline{\H}_{\g(v_0)}
\big(\mu^+[v_0],-I[v_0]-1\big)\Big]\ \ \ \ \text{satisfying}\ \ \ \ \g(v_0)<g $$
and outlying vertex factors
\begin{equation}\label{qqq}
\Big[\overline{\H}_{\g(v)}\big(\mu^+[v],I[v]-1\big)\Big]
\end{equation}
satisfying {either} $\g(v)<g$ {or} $\g(v)=g$.
If $\g(v)=g$, then the entire genus of $\Gamma$ is concentrated on $v$
and there can be no other outlying vertices (or else the
contribution of $\Gamma$ vanishes)
and only a single edge. By stability, at least two parts
of $\mu^+$ must be distributed to $v_0$, so $\n(v)\leq n$.
Hence, the outlying vertex factors \eqref{qqq} satisfy
{either} $\g(v)<g$ or
\begin{equation} \label{hhyy}
\g(v)=g \ \ \ \text{and} \ \ \ \n(v)\leq n\, .
\end{equation}

We study now all the contributions of
graphs ${\Gamma\in \mathsf{S}^\star_{g,\mu^+}}$
which carry an outlying vertex ${v\in \VV_{\mathrm{out}}(\Gamma)}$
satisfying
\begin{equation*}
\g(v)=g \ \ \ \text{and} \ \ \ \n(v)= n\, ,
\end{equation*}
 the extremal case of \eqref{hhyy}.
We have seen
\begin{equation} \label{w98}
 \text{\em $\Gamma=\{v_0,v_1\}$ with a single edge $e$ and
$\g(v_0)=0$, $\g(v_1)=g$}
\end{equation}
is the only possibility for a nonvanishing contribution.
By definition, the negative part of $\mu^+$ must be
distributed to $v_0$. In order for $\n(v_1)=n$, exactly
one positive part $m^+_i$ of $\mu^+$ must also be distributed to
$v_0$. Let $\Gamma_i \in \mathsf{S}^\star_{g,\mu^+}$ be the simple
star graph of type \eqref{w98} with
marking distribution
\begin{equation}\label{w99}
\{m^+_i,-1\} \mapsto v_0\, , \ \ \ \
\{m^+_1,\ldots,\widehat{m^+_i},\ldots, m^+_n\} \mapsto v_1\, .
\end{equation}
The
graphs $\Gamma_1,\, \Gamma_2,\, \ldots,\, \Gamma_n$
are the only elements of $\mathsf{S}^\star_{g,\mu^+}$
which  saturate the bounds \eqref{hhyy} and have possibly
nonvanishing contributions to the right side of \eqref{tttzzz}.

The negative part of $\mu^+$ corresponds to the last marking
of the associated moduli space $\overline{\mathcal{M}}_{g,n+1}$. Let
$$\epsilon: \overline{\mathcal{M}}_{g,n+1} \rightarrow
\overline{\mathcal{M}}_{g,n}$$
be the map forgetting the last marking.
We push-forward formula \eqref{tttzzz} under $\epsilon$,
\begin{multline} \label{tttzzze}
\epsilon_*\mathsf{H}_{g,\mu^+} =
\sum_{\Gamma\in \mathsf{S}^\star_{g,\mu^+}}
\sum_{I\in {\mathrm{Tw}}(\Gamma)}
\frac {\prod_{e\in \mathsf{E}(\Gamma)} I(e)}{|{\text{Aut}}(\Gamma)| }
\;
\epsilon_*\xi_{\Gamma*}\Bigg[
\Big[\overline{\H}_{\g(v_0)}
\big(\mu^+[v_0],-I[v_0]-1\big)\Big]
\\
\cdot \prod_{v\in \VV_{\mathrm{out}}(\Gamma)}
\Big[\overline{\H}_{\g(v)}\big(\mu^+[v],I[v]-1\big)\Big]\Bigg] \,
\end{multline}
to obtain an equation in $A^{g-1}(\overline{\mathcal{M}}_{g,n})$.

We study the precise contribution
to the right side of \eqref{tttzzze} of the
graphs $\Gamma_i$
characterized by \eqref{w98} and \eqref{w99}.
The graph $\Gamma_i$ has a unique possible twist{\footnote{We use the condition
$m_i^+>0$ here.}}
$$I(e)=m^+_i\ .$$
The contribution of $\Gamma_i$ is
\begin{equation*}
m^+_i\cdot \epsilon_*\xi_{\Gamma_i*}\Bigg[
\Big[\overline{\H}_{0}
\big(m^+_i,-1,-m^+_i-1\big)\Big]  \cdot
\Big[\overline{\H}_{g}\big(m^+_1,\ldots,m^+_{i-1},m^+_i-1,m^+_{i+1}, \ldots m^+_n)\Big]\Bigg]
\end{equation*}
where we have
\begin{equation} \label{jjed}
\epsilon \circ \xi_{\Gamma_i}:
\overline{\mathcal{M}}_{0,3}
\times \overline{\mathcal{M}}_{g,n} \stackrel{\sim}{\longrightarrow}
\overline{\mathcal{M}}_{g,n} \, .
\end{equation}
Using the isomorphism \eqref{jjed}, the contribution of $\Gamma_i$ is
simply
$$\text{Cont}(\Gamma_i) = m_i^+\cdot \Big[\overline{\H}_{g}\big(m^+_1,
\ldots,m^+_{i-1},m^+_i-1,m^+_{i+1}, \ldots m^+_n)\Big]
\in A^{g-1}(\overline{\mathcal{M}}_{g,n})\ .
$$

The contribution of $\Gamma_n$ to the right side of \eqref{tttzzze}
is special. Since by construction $m^+_n-1=m_n$, we see
$$\text{Cont}(\Gamma_n) = (m_n+1)\cdot \Big[\overline{\H}_{g}\big(m_1,
\ldots, m_n)\Big] = (m_n+1) \cdot [\overline{\H}_{g}\big(\mu)]
\in A^{g-1}(\overline{\mathcal{M}}_{g,n})\,
$$ with $m_n+1\neq 0$.
The contributions of the graphs $\Gamma_1, \ldots, \Gamma_{n-1}$
are proportional to  classes
$$[\overline{\H}_{g}\big(\mu')]\in
A^{g-1}(\overline{\mathcal{M}}_{g,n})
$$
for non-negative partitions $\mu'$ with largest
part {\em larger} than the largest part of $\mu$.

\vspace{8pt}
\noindent{\bf Lemma A2.}
{\em Let $\mu=(m_1,\ldots,m_n)$ be holomorphic with no $0$ parts.
The class
$$[\overline{\H}_g(\mu)] \in A^{g-1}(\oM_{g,n})$$ is
determined via formula \eqref{tttzzze}
by classes of the following four types:
\begin{enumerate}
\item[(i)] $\mathsf{H}_{g,\mu^+}\in A^g(\oM_{g,n})$,
\item[(ii)] $[\overline{\H}_{g'}(\mu')]\in A^{g'}(\oM_{g',n'})$ for
$(g',n') \stackrel{\circ}{<}(g,n)$ and $\mu'=(m'_1,\ldots,m'_{n'})$ strictly meromorphic,
\item[(iii)] $[\overline{\H}_{g''}(\mu'')]\in A^{g''-1}(\oM_{g'',n''})$
for $(g'',n'')\stackrel{\circ}{<} (g,n)$ and $\mu''=(m''_1,\ldots,m''_{n''})$
holomorphic.
\item[(iv)] $[\overline{\H}_{g}(\mu'')]\in A^{g-1}(\oM_{g,n})$
for $\mu''=(m''_1,\ldots,m''_{n})$
holomorphic with the largest part of $\mu''$ {\em larger}
than the largest part of $\mu$.
\end{enumerate}}

\vspace{8pt}
\begin{proof}
The $\epsilon$ push-forward of the class (i) appears on the left of
\eqref{tttzzze}. Nontrivial simple star graphs
not equal to $\Gamma_1,\ldots,\Gamma_n$
contribute
(ii) via the center vertex and (iii) via the outlying vertices.
The graphs $\Gamma_1,\ldots, \Gamma_{n-1}$ contribute (iv).
We then solve \eqref{tttzzze}
 for the contribution
$(m_n+1) \cdot [\overline{\H}_{g}\big(\mu)]$
of  $\Gamma_n$ .
\end{proof}

\vspace{8pt}

\noindent{A.6.3\, \bf{Determination.}}
We now combine Lemmas A1 and A2 to obtain the following
basic result.

\vspace{8pt}
\noindent{\bf Theorem A3.}
{\em Conjecture A effectively
determines the classes
$$[\overline{\H}_{g}(\mu)] \in A^*(\overline{\mathcal{M}}_{g,n})$$
both in the holomorphic and the strictly meromorphic cases.}

\vspace{8pt}
\begin{proof}
If $g=0$, the class $[\overline{\H}_{g}(\mu)]$ vanishes
in the holomorphic case and is the indentity $1\in A^*(\overline{\mathcal{M}}_{0,n})$ in the strictly meromorphic case.
We proceed by induction with respect to the partial ordering
$\stackrel{\circ}{<}$ on pairs defined by \eqref{kk12}.

If $\mu$ is strictly meromorphic, we apply Lemma A1.
Conjecture A determines the class
$$\mathsf{H}_{g,\mu} \in
A^{g}(\overline{\mathcal{M}}_{g,n})\, .$$
The rest of the classes specified by Lemma A1
are {\em strictly lower} in the partial ordering
$\stackrel{\circ}{<}$ .


If $\mu$ is holomorphic and $\mu$ has parts equal to $0$,
then either
$$g=1\, , \ \ \mu=(0,\ldots,0)\, ,  \ \ \text{and}\ \   [\overline{\H}_{g}(0,\ldots,0)]=1 \in A^0(\overline{\mathcal{M}}_{1,n})$$
or $\mu'$ (obtained by removing the $0$ parts) yields a
class
$$[\overline{\H}_{g}(\mu')] \in
A^{g-1}(\overline{\mathcal{M}}_{g,n'})\, $$
which is {\em strictly lower} in the partial ordering
$\stackrel{\circ}{<}$ . We then
apply the pull-back \eqref{j123} to determine
$[\overline{\H}_{g}(\mu)]\in \overline{\mathcal{M}}_{g,n}$.

We may therefore assume $\mu$ is holomorphic with no $0$ parts.
We then apply Lemma A2.
Conjecture A determines the class
$$\epsilon_*\mathsf{H}_{g,\mu^+} \in
A^{g-1}(\overline{\mathcal{M}}_{g,n})\, .$$
The rest of the classes specified by Lemma A2
either are
{\em strictly lower}
in the partial ordering $\stackrel{\circ}{<}$ or
are
\begin{enumerate}
\item[(iv)] $[\overline{\H}_{g}(\mu'')]\in A^{g-1}(\oM_{g,n})$
for $\mu''=(m''_1,\ldots,m''_{n})$
holomorphic with the largest part of $\mu''$ {\em larger}
than the largest part of $\mu$.
\end{enumerate}

To handle (iv),
we
apply  descending induction on the largest part of $\mu$ in the holomorphic
case.
The base  for the descending induction occurs
when the largest part is $2g-2$, then there are {\em no partitions
with larger largest part}.
\end{proof}

\vspace{8pt} We have presented a calculation of
the classes of the closures
$$\overline{\H}_g(\mu) \subset \overline{\mathcal{M}}_{g,n}$$
of the
moduli spaces of canonical divisors on nonsingular curves via Conjecture A
{\em and} the virtual components of the moduli space of twisted canonical divisors. Theorem A3 yields an effective method with result
$$[\overline{\H}_g(\mu)] \in R^*(\overline{\mathcal{M}}_{g,n})\, .$$

The $g=2$ cases with $\mu=(3,-1)$ and $\mu=(2,1,-1)$
 discussed in Section A.5  may also be viewed
as steps in the recursion of Theorem A3 to calculate
\begin{equation}\label{dxx23}
\overline{\H}_{2,(2)} \in R^1(\overline{\mathcal{M}}_{2,1}) \ \ \ \text{and}
\ \ \  \overline{\H}_{2,(1,1)}\in R^1(\overline{\mathcal{M}}_{2,2})
\end{equation}
respectively. Since we already have formulas for the classes \eqref{dxx23},
the calculations serve to check Conjecture A.


\end{document}